\documentclass[reqno]{amsart}

\usepackage{amsmath,amssymb,graphicx
}
\usepackage{verbatim} 

\usepackage[latin1]{inputenc}
\usepackage{enumerate}
\usepackage{hyperref}

\graphicspath{{Figs/}}
\newtheorem{theorem}{Theorem}[section]
\newtheorem{proposition}[theorem]{Proposition}
\newtheorem{corollary}[theorem]{Corollary}
\newtheorem{lemma}[theorem]{Lemma}
\theoremstyle{definition}
\newtheorem{remark}[theorem]{Remark}
\newtheorem{definition}[theorem]{Definition}

\newcommand{\Motzbic}{M}
\newcommand{\Motz}{M^*}
\newcommand{\Pos}{Q}
\newcommand{\WPos}{Q^*}
\newcommand{\Gen}{G}
 
\newcommand{\Motzbicset}{{\mathcal M}}
\newcommand{\Motzset}{{\mathcal M}^*}
\newcommand{\Posset}{{\mathcal Q}}

\newcommand{\Genset}{\mathcal{G}}
\newcommand{\E}{\mathcal{E}}
\newcommand{\Cylset}{\mathcal{O}}
\newcommand{\Cyl}{O}

\newcommand\touch[1]{\check{{#1}}}

\newcommand{\haut}{\operatorname{ht}}
\newcommand{\TL}{\operatorname{TL}}

\renewcommand{\H}{\operatorname{Heap}}

\newcommand{\Aaff}{\aff{A}}

\newcommand{\aff}[1]{\widetilde{#1}}

\allowdisplaybreaks

\begin{document}

\title[Fully commutative elements  in finite and affine Coxeter groups]{Fully commutative elements in finite and affine Coxeter groups}

\author[Riccardo Biagioli]{Riccardo Biagioli}
\address{Institut Camille Jordan, Universit\'e Claude Bernard Lyon 1,
69622 Villeurbanne Cedex, France}
\email{biagioli@math.univ-lyon1.fr}
\urladdr{http://math.univ-lyon1.fr/{\textasciitilde}biagioli}

\author[Fr\'ed\'eric Jouhet]{Fr\'ed\'eric Jouhet}
\address{Institut Camille Jordan, Universit\'e Claude Bernard Lyon 1,
69622 Villeurbanne Cedex, France}
\email{jouhet@math.univ-lyon1.fr}
\urladdr{http://math.univ-lyon1.fr/{\textasciitilde}jouhet}

\author[Philippe Nadeau]{Philippe Nadeau}
\address{CNRS, Institut Camille Jordan, Universit\'e Claude Bernard Lyon 1,
69622 Villeurbanne Cedex, France}
\email{nadeau@math.univ-lyon1.fr}
\urladdr{http://math.univ-lyon1.fr/{\textasciitilde}nadeau}


\date{\today}

\subjclass[2010]{}


\begin{abstract}
An element of a Coxeter group $W$ is {\em fully commutative} if any two of its reduced decompositions are related by a series of transpositions of adjacent commuting generators. These elements were extensively studied by Stembridge, in particular in the finite case. They index naturally a basis of the generalized Temperley--Lieb algebra. In this work we deal with any finite or affine Coxeter group $W$, and we give explicit descriptions of fully commutative elements. Using our characterizations we then enumerate these elements according to their Coxeter length, and find in particular that the corrresponding growth sequence is ultimately periodic in each type.  When the sequence is infinite, this implies that the associated  Temperley--Lieb algebra has linear growth. 
\end{abstract}

\maketitle


\section*{Introduction}
\label{sec:intro}

Let $W$ be a Coxeter group. An element $w \in W$ is said to be {\em fully commutative} if any reduced expression for $w$ can be obtained from any other one by transposing adjacent pairs of commuting generators. These elements were extensively studied by Stembridge in the series of papers~\cite{St1}-\cite{St3}.  He classified in~\cite{St1} the Coxeter groups having a finite number of fully commutative elements, which is independently done by Graham in~\cite{Graham}, or Fan in~\cite{Fan} in the simply--laced case.  Stembridge also gives in \cite{St1} a useful characterizing property of full commutativity, by showing that reduced words for such an element can be viewed as the linear extensions of a {\em heap}, which is a certain kind of poset whose vertices are labeled by generators of $W$. In \cite{St3}, Stembridge enumerates fully commutative elements for each of the previous finite cases, while  connections with enriched $P$-partitions (the letter $P$ stands here for ``Poset") and Schur's $Q$-functions are studied  in \cite{St2} in types $B$ and $D$.
 
The original context for the appearance of full commutativity is algebraic and relates to the \emph{generalized Temperley--Lieb algebras}.  The (type $A$) \emph{Temperley--Lieb algebra} was first defined in \cite{TemperleyLieb}, in the context of statistical mechanics. Later, it was realized by Jones in~\cite{JonesAnnals} that it is a quotient of the \emph{Iwahori--Hecke algebra} of type $A$. This point of view was used by Fan in~\cite{Fan} in the simply--laced case, and by Graham in~\cite{Graham} in general, to define a generalized Temperley--Lieb algebra for each Coxeter group. They proved that for any $W$, the associated generalized Temperley--Lieb algebra admits a linear basis indexed by the fully commutative elements of $W$.

 The set of fully commutative elements was also studied in connection to~\emph{Kazhdan--Lusztig cells}, which form a partition of the Coxeter group $W$.  In~\cite{GreLoFCKL}, Green and Losonczy characterize when the set of fully commutative elements of a finite $W$ is a union of double-sided cells. This was extended to affine types in the works~\cite{ShiFCKL1,ShiFCKL2} of Shi. Other cells were defined and studied by Fan~\cite{Fan_JAMS} and Fan and Green~\cite{FanGreen_Affine}\\

The main goal of the present paper is to give a complete description, in terms of heaps, of fully commutative elements for each affine irreducible Coxeter group. From this, we derive the  generating functions of the fully commutative elements, according to the Coxeter length, for all finite and affine Coxeter groups: this extends naturally the work~\cite{St3} of Stembridge.  

Our main characterization result is performed in a case-by-case fashion, so that it is split into several statements. These are Theorem~\ref{theorem:walks_type_Atilde} in type $\Aaff$,  Theorems~\ref{theo:affineCfamilles} and~\ref{theo:affineBDfamilles}  for the other classical types, and the results of Section~\ref{sub:affExcept} in exceptional  types. It can be summarized as follows.\\

\noindent {\bf Main result}: {\em For each affine irreducible Coxeter group $W$, we give an explicit description of its subset $W^{FC}$ of fully commutative elements.}\\
 
Here, by ``description", we mean that for each classsical type we define precisely disjoint families of heaps partitioning $W^{FC}$, and whose generating functions can all be computed.  Our characterizations naturally bring us to define a particular subfamily of elements in $W^{FC}$, which we call \emph{alternating} fully commutative elements, and to which we can associate bijectively certain \emph{lattice walks}. The latter are crucial for computing generating functions, as the various classes of lattice walks we are concerned with can all be decomposed by taking the Coxeter length into account.

 We therefore derive an expression for $W^{FC}(q):=\sum_{w \in W^{FC}}q^{\ell(w)}$, when $W$ is any finite or affine irreducible Coxeter group, and $\ell$ denotes the Coxeter length. This results in the following formal consequence, corresponding to Corollary~\ref{cor:Atilde} (\emph{resp.} Proposition~\ref{typeA}) in type $\Aaff$ (\emph{resp.} $A$) and Propositions~\ref{prop:gf_affC}-\ref{prop:afftypeD} (\emph{resp.} Proposition~\ref{typeB}) in types $\aff{C}$, $\aff{B}$ and $\aff{D}$ (\emph{resp.} $B$ and $D$), while the results for exceptional types are given in Section~\ref{sub:affExcept}.\\

\noindent {\bf Corollary A}:
{\em For each finite or affine group $W$, we compute the generating function $W^{FC}(q)$.}\\

For the finite type $A$, this recovers the nice $q$-analogue of the Catalan numbers  discovered by Barcucci et al. in~\cite{BDPR}, where they enumerated $321$-avoiding permutations (corresponding here to the fully commutative elements) according to their inversion number. In the case of the affine symmetric group $\Aaff$, a similar $q$-analogue has been recently found by Hanusa and Jones in~\cite{HanJon}; from our analysis we  find a much simpler expression than in~\cite{HanJon} for the corresponding generating function. To our knowledge, for all other classical affine or finite types our results are new. 

Moreover,  thanks to our interpretations in terms of lattice walks, a striking consequence of our characterizations above is that, for any affine Coxeter group $W$, the {\em growth sequence} of  $W^{FC}$ (the sequence of coefficients of $W^{FC}(q)$) is \emph{ultimately periodic}. This was already shown in type $\Aaff$ by Hanusa and Jones (see~\cite[\S 5]{HanJon}); however our method  settles positively a question in~\cite{HanJon} regarding the beginning of the periodicity. These results correspond to Theorem~\ref{theo:periodicityA} in type $\Aaff$ and Theorem~\ref{theo:periodicityBCD} in types $\aff{C}$, $\aff{B}$ and $\aff{D}$, and appear in Section~\ref{sub:affExcept} for exceptional types.\\

\noindent {\bf Corollary B}:
{\em For each irreducible, affine group $W$, the growth sequence of $W^{FC}$ is {\em ultimately periodic}, with period recorded in the following table ($\aff{F}^{FC}_4,\aff{E}^{FC}_8$ are finite sets):}
  \[\begin{array}{ l || c|c|c|c|c|c|c|c}
    \textsc{Affine Type} &\aff{A}_{n-1}&\aff{C}_n&\aff{B}_{n+1}&\aff{D}_{n+2}&\aff{E}_6&\aff{E}_7&\aff{G}_2&\aff{F}_4,\aff{E}_8 \\ \hline
    \textsc{Periodicity} &n&n+1&(n+1)(2n+1)&n+1&4&9&5&1 \\
 \end{array}\]

 In addition, Theorems~\ref{theo:periodicityA} and~\ref{theo:periodicityBCD} also give the exact Coxeter length at which the periodicity begins for classical types, while the corresponding result for exceptional types are recorded in Section~\ref{sub:affExcept}. From the generating functions, it is also possible to determine, for each affine Coxeter group $W$, the {\em mean value $\mu_W$} of the growth sequence, which is the arithmetic mean of the values over a period, or equivalently the limit of the arithmetic means of the $k$ first values of the growth sequence when $k$ tends to infinity. These are given by Propositions~\ref{prop:meanA} and~\ref{prop:meanBCD} in classical types, and easily computed from the results of Section~\ref{sub:affExcept} in exceptional types.\\

Finally, the periodicity result has an interesting consequence in terms of Temperley-Lieb algebras. Recall that the growth of an algebra, with respect to a finite set of generators, is the function which associates to an integer $\ell$ the dimension of the vector space generated by all products of at most $\ell$ generators.\\

\noindent {\bf Corollary C}:
{\em For any irreducible affine Coxeter group with infinitely many fully commutative elements, the corresponding Temperley--Lieb algebra has {\em linear growth.}}\\

We explain and refine this in detail in Section~\ref{sub:TL}. As we also explain in the same section, this algebraic viewpoint allowed us to use computer tools to deal with fully commutative elements. In particular, the package \texttt{GBNP} for \texttt{GAP}~\cite{GAP4}, which uses noncommutative Gr\"obner basis techniques, was used intensively.

\section{Fully commutative elements, heaps and walks}
\label{sec:preliminaries}

In this first section, we will define fully commutative elements in Coxeter groups. These elements can be regarded as certain commutation classes of words. Therefore, following the approach of Stembridge~\cite{St1}, our central tool in studying fully commutative elements will be the structure of heaps, originally defined by Viennot in~\cite{ViennotHeaps}. In Section~\ref{sub:alternating}, we identify a particular class of these elements which occur naturally in classical types, and encode them bijectively  by lattice walks. We recall finally the classification of finite and affine Coxeter groups in Section~\ref{sub:finiteaffine}.

\subsection{Fully commutative elements in Coxeter groups}
\label{sub:fullycomm}

Let $M$ be a square symmetric matrix indexed by a finite set $S$, satisfying $m_{ss}=1$ and, for $s\neq t$, $m_{st}=m_{ts}\in\{2,3,\ldots\}\cup\{\infty\}$. The {\em Coxeter group} $W$ associated with the  Coxeter matrix $M$ is defined by generators $S$ and relations $(st)^{m_{st}}=1$ if $m_{st}<\infty$. These relations can be rewritten more explicitly as $s^2=1$ for all $s$, and \[\underbrace{sts\cdots}_{m_{st}}  = \underbrace{tst\cdots}_{m_{st}},\]  where $m_{st}<\infty$, the latter being called \emph{braid relations}. When $m_{st}=2$, they are simply \emph{commutation relations} $st=ts$, sometimes referred to as short braid relations in the literature. 

The {\em Coxeter graph} $\Gamma$ associated to $(W,S)$ is the graph with vertex set $S$ and, for each pair $\{s,t\}$ with $m_{st}\geq 3$, an edge between $s$ and $t$ labeled by $m_{st}$. When $m_{st}=3$ the edge is usually left unlabeled since this case occurs frequently. Therefore non adjacent vertices correspond precisely to commuting generators.

For $w\in W$, the {\em length} of $w$, denoted by $\ell(w)$, is the minimum length $l$ of any expression $w=s_1\cdots s_l$ with $s_i\in S$. These expressions of length $\ell(w)$ are called \emph{reduced}, and we denote by $\mathcal{R}(w)$ the set of all reduced expressions of $w$. A fundamental result in Coxeter group theory, sometimes called the {\em Matsumoto property}  is that any expression in $\mathcal{R}(w)$ can be obtained from any other one using only braid relations (see for instance~\cite{Humphreys}). The notion of full commutativity is a strengthening of this property.

\begin{definition}
\label{defi:FC}
 An element $w$ is \emph{fully commutative} (FC) if any reduced expression for $w$ can be obtained from any other one by using only commutation relations.
\end{definition}

The following characterization of FC elements, originally due to Stembridge, is particularly useful in order to test whether a given element is FC or not.

\begin{proposition} [Stembridge \cite{St1}, Prop. 2.1]
\label{prop:caracterisation_fullycom}
An element $w\in W$ is fully commutative if and only if for all $s,t$ such that $3\leq m_{st}<\infty$, there is no expression in $\mathcal{R}(w)$ that contains the factor $\underbrace{sts\cdots}_{m_{st}}$.
\end{proposition}

\begin{proof}
It is clear that the condition is sufficient. Now assume that $\mathbf{w}\in\mathcal{R}(w)$ has such a factor, then we get another reduced expression $\mathbf{w}'$ by applying the corresponding braid relation. But commutation relations cannot modify the subword of $\mathbf{w}$ in $s$ and $t$, so $\mathbf{w}'$ is not in the same commutation class and $w$ is not fully commutative.
\end{proof}

Therefore an element $w$ is FC if all reduced words avoid all braid words; since, by definition, $\mathcal{R}(w)$ forms a commutation class, the concept of heaps helps to capture the notion of full commutativity.

\subsection{Heaps and fully commutative elements}
\label{sub:heapsFC}

\subsubsection{Heaps}
\label{subsub:heaps}

We use the definition as given in~\cite[p.20]{GreenBook} (see also \cite[Definition 2.2]{KrattHeaps}), and discuss it afterwards.

\begin{definition}[Heap]
\label{defi:heaps}
Let $\Gamma$ be a finite graph with vertex set $S$. A {\em heap} of type $\Gamma$ (or $\Gamma$-heap) is a finite poset $(H,\leq)$, together with a labeling map $\epsilon:H\to\Gamma$, which satisfies the following conditions:
\begin{enumerate}[(a)]
\item \label{it:heap1} For any vertex $s$, the subposet $H_s:=\epsilon^{-1}(\{s\})$ is totally ordered, and for any edge $\{s,t\}$, the subposet $H_{\{s,t\}}:=\epsilon^{-1}(\{s,t\})$ is totally ordered.
\item \label{it:heap2} The ordering on $H$ is the transitive closure of the relations given by all chains $H_s$ and $H_{\{s,t\}}$, i.e. the smallest partial ordering containing these chains.
\end{enumerate}
\end{definition}

\begin{remark}
There is a more concrete way to reformulate condition \eqref{it:heap2}: one has $x< y$ in $H$ if and only if there exist $x_0=x< x_1< \cdots < x_{k-1}< x_k=y$ such that $\epsilon(x_i),\epsilon(x_{i+1})\in S$ are equal or adjacent in the graph $\Gamma$ for all $i$.
\end{remark}

Two $\Gamma$-heaps are {\em isomorphic } if there exists a poset isomorphism between them which preserves the labels.
The size $|H|$ of a heap $H$ is its cardinality. Given any subset $I\subset S$, we will note $H_{I}$ the subposet induced by all elements of $H$ with labels in $I$. In particular $H_{\{s\}}$ is the chain  $H_{s}=s^{(1)}<s^{(2)}<\cdots<s^{(k)}$ where $k=|H_s|$ is its cardinality. If $s,t$ are two labels such that $m_{st}\geq 3$, note that $H_{\{s,t\}}$ is also a chain.

Let us point out a simple operation on heaps: if $(H,\leq,\epsilon)$ is a heap, then its {\em dual heap} is $(H,\geq,\epsilon)$, which is the heap with the inverse order where the labels are kept the same. It is clear that this is indeed a heap as in Definition~\ref{defi:heaps}.

\subsubsection{Heaps and words}
\label{subsub:heapswords}

 Fix a word $\mathbf{w}=s_{a_1}\cdots s_{a_l}$ in $S^*$, the free monoid generated by $S$. Define a partial ordering $\prec$ of the index set $\{1,\ldots, l\}$ as follows: set $i\prec j$ if $i<j$ and $\{s_{a_i},s_{a_j}\}$ is an edge of $\Gamma$, and extend by transitivity. We denote by $\H({\mathbf{w}})$ this poset together with $\epsilon:i\mapsto s_{a_i}$. It is easy to see that this indeed forms a heap as in Definition~\ref{defi:heaps}. Also, words from the same commutation class are sent to isomorphic heaps.

\begin{proposition}[Viennot, \cite{ViennotHeaps}]
\label{prop:wordtoheap} 
 Let $\Gamma$ be a finite graph. The map $\mathbf{w}\to \H({\mathbf{w}})$ induces a bijection between $\Gamma$-commutation classes of words, and  finite $\Gamma$-heaps (up to isomorphism). 
\end{proposition}

The inverse bijection goes as follows: a \emph{linear extension} of a poset $H$ is a linear ordering $(h_1<\cdots <h_l)$ of the elements of $H$ such that $h_i<h_j$ implies $h_i\prec h_j$. Given a heap $(H,\epsilon)$ and a linear extension $h_1,\ldots,h_l$, construct the word $\epsilon(h_1)\cdots \epsilon(h_l)$  in $S^*$. Applying this to all such linear extensions gives the wanted commutation class.

In Figure~\ref{fig:wordheap}, we fix a (Coxeter) graph on the left, and we give two examples of words with the corresponding heaps. In the Hasse diagram of $\H(\mathbf{w})$, elements with the same labels will be drawn in the same column.

\begin{figure}[!ht]
\begin{center}
 \includegraphics[height=3cm]{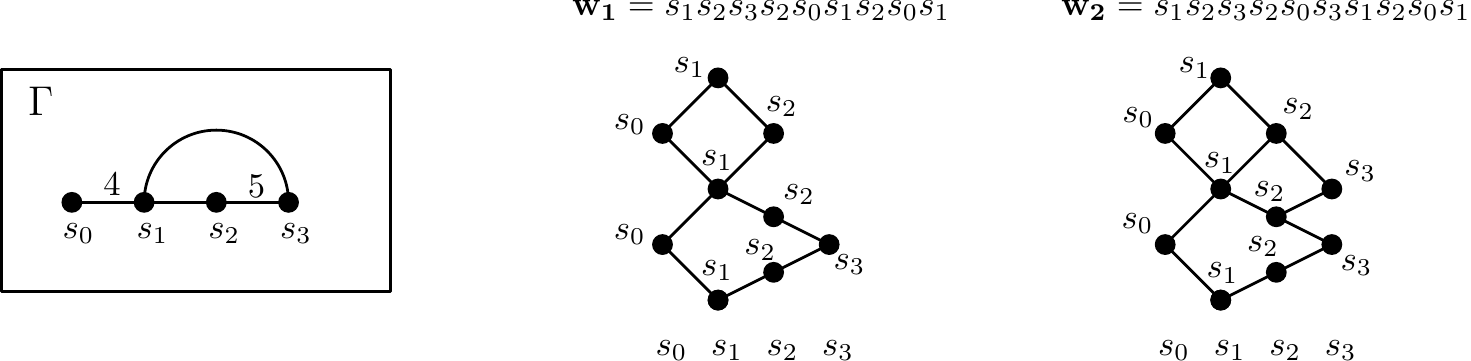}
\caption{\label{fig:wordheap} Two words and their respective heaps.}
\end{center}
\end{figure}

\subsubsection{Fully commutative heaps}
We now fix a Coxeter group $W$ with Coxeter diagram $\Gamma$. When $w$ is a fully commutative element, the heaps of its reduced words are all isomorphic by Proposition~\ref{prop:wordtoheap}, so we can give the following definition.

\begin{definition} If $w$ is a FC element and $\mathbf{w}$ is a reduced word for it, define $\H(w):=\H(\mathbf{w})$. Heaps of this form will be called {\em FC heaps}.
\end{definition}

For a FC heap $H=\H(w)$, linear extensions of $H$ are in bijection with reduced words for $w$. Say that a chain $i_1\prec \cdots \prec i_m$ in a poset $H$ is {\em convex} if the only elements $u$ satisfying $i_1\preceq u\preceq i_m$ are the elements $i_j$ of the chain. The next result gives an intrinsic characterization of {\em FC heaps}.

\begin{proposition}[Stembridge, \cite{St1}, Proposition 3.3]
\label{prop:heaps_fullycom}
Let $\Gamma$ be a Coxeter graph. A $\Gamma$-heap $H$ is FC if and only if the following two conditions are verified:
\begin{itemize}
 \item[$(a)$] there is no convex chain $i_1\prec\cdots\prec i_{m_{st}}$ in $H$ such that $\epsilon(i_1)=\epsilon(i_3)=\cdots=s$ and   $\epsilon(i_2)=\epsilon(i_4)=\cdots=t$ where $3\leq m_{st}<\infty$; 
 \item[$(b)$] there is no covering relation $i\prec j$ in $H$ such that $s_{i}=s_{j}$.
 \end{itemize}
\end{proposition}

The heap on the right of Figure~\ref{fig:wordheap} is a FC heap, whereas the one on the left is not since it contains the convex chain with labels $(s_2,s_1,s_2)$ while $m_{s_1 s_2}=3$.

\begin{corollary} Let $(W,S)$ be a Coxeter system with Coxeter graph $\Gamma$. FC
elements of $W$ are in bijection with $\Gamma$-heaps verifying the two conditions of Proposition~\ref{prop:heaps_fullycom}. 
 \end{corollary}

In the next section we will exhibit a class of heaps which play an important role for the affine Coxeter systems we will study.
 
\subsection{Alternating heaps and walks}
\label{sub:alternating}

\begin{definition}
\label{defi:alternating}
Consider a graph $\Gamma$, and a $\Gamma$-heap $H$. We say that $H$ is {\em alternating} if for each edge $\{s,t\}$ of $\Gamma$, the chain $H_{\{s,t\}}$ has alternating labels $s$ and $t$ from bottom to top.
\end{definition}

A word $\mathbf{w} \in S^*$ is \emph{alternating} if, for each edge $\{s,t\}$, the occurrences of $s$ alternate with those of $t$. Clearly $\mathbf{w}$ is alternating if and only if $\H(\mathbf{w})$ is.
\medskip

Now let $\Gamma=\Gamma(W,S)$ be a Coxeter graph. Say that a FC element $w$ is alternating if $\H(w)$ is. By using Proposition~\ref{prop:heaps_fullycom} we obtain the following characterization. 

\begin{proposition}\label{prop:altFCheaps}
Let $H$ be an alternating $\Gamma$-heap. Then H is FC if and only if, for each vertex $s \in\Gamma$ of degree $1$, $H$ has at most one $s$-element or $m_{st}>3$ for the unique $t$ adjacent to $s$.
\end{proposition}

In the rest of this section, we fix $m_{v_0v_1},m_{v_1v_2},\ldots, m_{v_{n-1}v_{n}}$ in the set $\{3,4,\ldots\}\cup\{\infty\}$ and we consider the Coxeter system $(W,S)$ corresponding to the {\em linear} Coxeter graph $\Gamma_n=\Gamma_n((m_{v_iv_{i+1}})_i)$ of Figure~\ref{fig:linear_diagram}.

\begin{figure}[!ht]
\begin{center}
\includegraphics{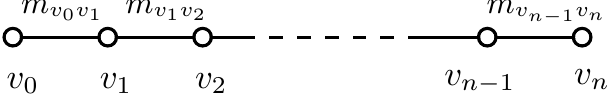}
\caption{\label{fig:linear_diagram} The linear Coxeter graph $\Gamma_n$.}
\end{center}
\end{figure}

An alternating heap is presented in Figure~\ref{fig:wordheappath}. It is FC if and only if $m_{01},m_{n-1n}>3$ by Proposition~\ref{prop:altFCheaps}. The advantage of having alternating heaps in the case of linear diagrams is that they have a nice encoding by walks which we will explore in the sequel.

\begin{figure}[!ht]
\begin{center}
\includegraphics[width=0.7\textwidth]{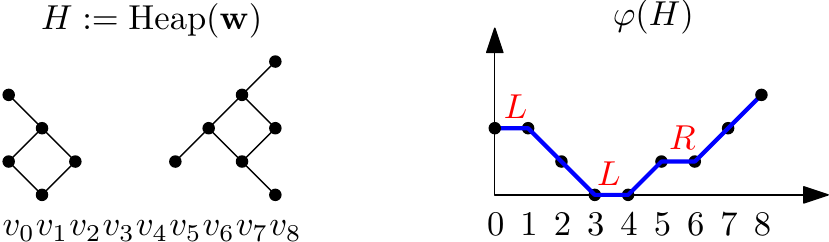}
\caption{\label{fig:wordheappath} The heap $H$ corresponding to the alternating word $\mathbf{w}=v_8v_5v_1v_0v_7v_8v_6v_2v_7v_1v_0v_8$ and its encoding by a walk.}
\end{center}
\end{figure}

\begin{definition}[Walks]
\label{defi:walks}
 A {\em walk} of length $n$ is a sequence $P=(P_0,P_1,\ldots,P_n)$ of points in $\mathbb{N}^2$ with its $n$ steps in the set $\{(1,1),(1,-1),(1,0)\}$, such that $P_0$ has abscissa $0$ and all horizontal steps $(1,0)$ are labeled either by $L$ or $R$. A walk is said to satisfy condition $(*)$ if all horizontal steps of the form $(i,0)\rightarrow (i+1,0)$ have label $L$. 
\end{definition}

The set of all walks of length $n$ will be denoted by $\Genset_n$. The subset of walks ending at $P_n=(n,0)$ will be denoted by $\Posset_n$, and the subset of $\Posset_n$ with $P_0=(0,i)$ will be denoted by $\Motzbicset^{(i)}_n$. For short we will write $\Motzbicset_n=\Motzbicset_n^{(0)}$ for walks starting and ending on the $x$-axis. To each family $\mathcal{F}_n \subseteq \Genset_n$ corresponds subfamilies $\mathcal{F}_n^*\subseteq\mathcal{F}_n$ consisting of those walks in $\mathcal{F}_n$ which satisfy the condition $(*)$, and $\touch{\mathcal{F}_n}\subseteq\mathcal{F}_n$ consisting of those walks which hit the $x$-axis at some point. 

\begin{remark}
\label{rem:MotzToDyck}
Write $U,D$ to represent steps $(1,1),(1,-1)$, and $L,R$ to represent steps $(1,0)$ with these labels; then we can encode our walks by the data of $P_0$ and a word in $\{U,D,L,R\}$. Consider the injective transformation on such words defined by  $U\mapsto UU,D\mapsto DD,L\mapsto UD,R\mapsto DU$.  It is well known and easy to prove that this transformation restricts to a bijection from $\Motzset_n$ to {\em length $2n$ Dyck walks}, i.e. walks from $(0,0)$ to $(2n,0)$ with steps $U,D$ staying above the $x$-axis. 
\end{remark}

  The \emph{total height} $\haut$ of a walk is the sum of the heights of its points: if $P_i=(i,h_i)$ then $\haut(P)=\sum_{i=0}^n h_i$. To each family $\mathcal{F}_n \subseteq \Genset_n $ we associate the series $F_n(q)=\sum_{P\in \mathcal{F}_n}q^{\haut(P)}$, and we define the generating functions in the variable $x$ by 
  
\[ F(x)=\sum_{n\geq 0} F_n(q) x^n, \quad F^*(x)=\sum_{n\geq 0} F^*_n(q) x^n\quad {\rm and} \quad \touch{F}(x)=\sum_{n\geq 0}\touch{ F}_n(q) x^n.\]

We now define a bijective encoding of alternating heaps by walks, which will be especially handy to compute generating functions in the next sections.

\begin{definition}[Map $\varphi$] 
\label{def:map_phi} Let $H$ be an alternating heap of type $\Gamma_n$. To each vertex  $v_i$ of $\Gamma_n$ we associate the point $P_i=(i,|H_{v_i}|)$. If $|H_{v_{i}}|=|H_{v_{i+1}}|>0$, we label the corresponding step by $L$ (\emph{resp.} $R$) if the lowest element of the chain $H_{\{v_i,v_{i+1}\}}$ has label $v_{i+1}$ (\emph{resp. $v_i$}). If $|H_{v_{i+1}}|=|H_{v_i}|=0$, we label the $i$th step by $L$. We define $\varphi(H)$ as the walk $(P_0,P_1,\ldots,P_n)$ with its possible labels.
\end{definition}

This is illustrated in Figure~

\begin{theorem}
\label{theorem:walk_encoding}
The map $H\mapsto \varphi(H)$ is a bijection between alternating heaps of type $\Gamma_n$ and $\Genset^*_n$. The size $|H|$ of the heap is the total height of $\varphi(H)$.
\end{theorem}
\begin{proof}
The map is clearly well defined since for any alternating heap $H$ and any $i\in\{0,\dots,n-1\}$, we have $-1 \leq |H_{v_i}| - |H_{v_{i+1}}| \leq 1$. Fix 
$(P_0,P_1,\ldots,P_n) \in \mathcal{G}_n^*$. If the step $P_i=(i,h_i) \rightarrow P_{i+1}=(i+1,h_{i+1})$ is equal to $(1,1)$ (\emph{resp.} $(1,-1)$), then we define a convex chain $\mathcal{C}_i$ of length $2h_{i}$ (\emph{resp.} $2h_{i+1}$) as $(v_{i+1},v_{i},\ldots, v_{i+1})$  (\emph{resp.} $(v_i,v_{i+1},\ldots, v_i)$).
If the step $P_i\rightarrow P_{i+1}$  is labeled by $L$ (\emph{resp.} $R$), then we define a convex chain  $\mathcal{C}_i$ of length $2h_i=2h_{i+1}$  as  $ (v_i,v_{i+1},\ldots, v_i,v_{i+1})$ (\emph{resp.} $(v_{i+1},v_{i},\ldots, v_{i+1},v_i)$). Next we define $H$ as the transitive closure of the chains $\mathcal{C}_0, \ldots,\mathcal{C}_{n-1}$. It is acyclic since $\Gamma_n$ is linear, so we get a heap $H$ which is uniquely defined, alternating, and satisfies $\varphi(H)=(P_0,P_1,\ldots,P_n)$. Since $h_i=|H_{v_i}|$ the result follows.
\end{proof}

\subsection{Finite and affine Coxeter groups}
\label{sub:finiteaffine}

The irreducible Coxeter systems corresponding to finite and affine Coxeter groups are completely classified (see~\cite{BjorBrenbook,Humphreys}), and the corresponding graphs are depicted in Figures~\ref{fig:finite_diagrams} and ~\ref{fig:affinediagrams}.

\begin{figure}[!ht]
\begin{center}
\includegraphics[width=\textwidth]{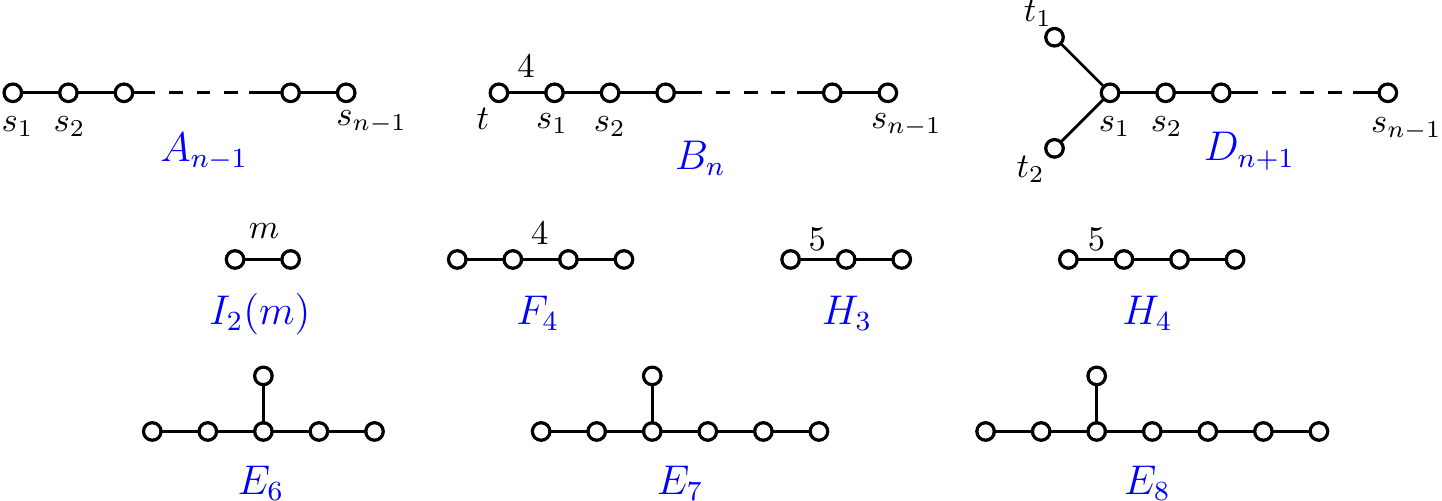}
\caption{\label{fig:finite_diagrams} Coxeter graphs for all irreducible finite types.}
\end{center}
\end{figure}
\begin{figure}[!ht]
\begin{center}
 \includegraphics[width=\textwidth]{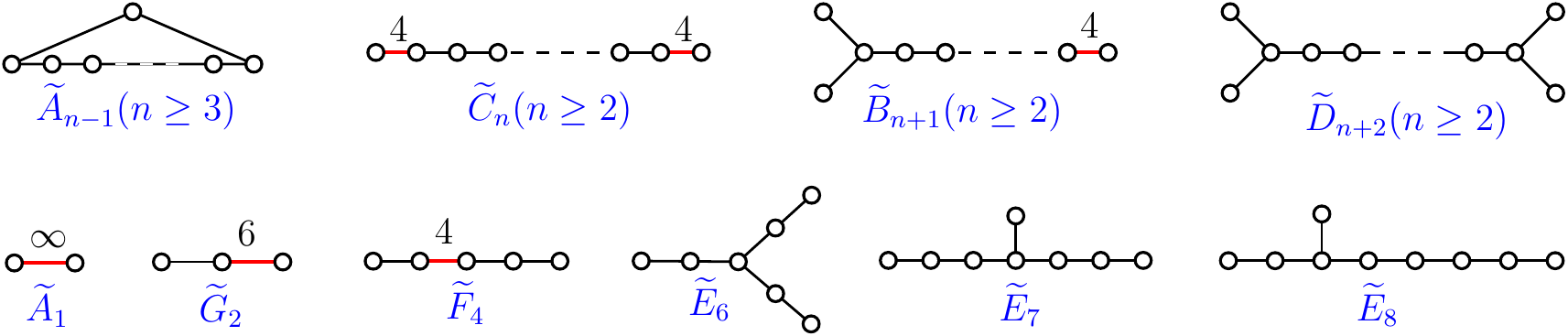}
 \caption{\label{fig:affinediagrams} Coxeter graphs for all irreducible affine types.}
\end{center}
\end{figure}

 For each Coxeter group $W$ corresponding to one of these types, we denote by $W^{FC}$ its set of FC elements and we define the generating function
 \[
 W^{FC}(q):=\sum_{w \in W^{FC}} q^{\ell(w)}.
\] 

The sequence of coefficients of the above series will be called the {\em growth sequence} of $W^{FC}$.

Section \ref{sec:a} deals with groups of types $\widetilde{A}_{n-1}$ and $A_{n-1}$. Sections~\ref{sec:bcd} and \ref{sec:enum} deal with the rest of the \emph{classical types}, which are the remaining infinite families:  it corresponds to affine types $\aff{C}_n,\aff{B}_{n+1},\aff{D}_{n+2}$ and finite types $B_n$ and $D_{n+1}$. The last remaining cases, which are the so-called \emph{exceptional types}, are studied in Section~\ref{sec:excep}, while Section~\ref{sec:further} will explore further works and questions related to our results, among which the application to Temperley--Lieb algebras.

\section[Type A affine]{Types $\widetilde{A}_{n-1}$ and $A_{n-1}$}
\label{sec:a}

In this section we study the case of FC elements in $W$ of type $\widetilde{A}_{n-1}$ and $A_{n-1}$. These correspond to the affine symmetric group and the symmetric group, respectively. As such, FC elements  were studied as they correspond to the so-called $321$-avoiding permutations in both cases, and the length function corresponds to the number of (affine) inversions.

\begin{center}
\includegraphics[width=0.2\textwidth]{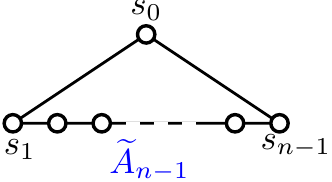}
\end{center}

We will here study FC elements through the use of heaps in the case $\widetilde{A}_{n-1}$, and give a nice characterization of them in terms of certain lattice walks. We assume $n\geq 3$ everywhere (the special case $\aff{A}_1$ is trivial as all its elements are FC).

\subsection{Affine permutations}

The generating function $\Aaff_{n-1}^{FC}(q)$ was first computed by Hanusa and Jones in~\cite{HanJon}. They use the following well-known realization of the Coxeter system of type $\Aaff_{n-1}$ as {\em affine permutations}. Consider the set of bijective transformations $\sigma:\mathbb{Z}\rightarrow\mathbb{Z}$ such that $\sigma(i+n)=\sigma(i)+n$ for all $i\in\mathbb{Z}$, as well as the normalization condition $\sum_{i=1}^n\sigma(i)=\sum_{i=1}^ni$. They form a group $\widetilde{S}_n$ under composition. It is known (see \cite{LusztigTransactions}) that this group is isomorphic to the Coxeter system of type $\Aaff_{n-1}$, via the only morphism extending $s_i\mapsto ((i,i+1))$ for $i=0,1,\ldots,n-1$, where $((i,i+1))$ is the affine permutation exchanging $i+kn$ and $i+1+kn$ for all $k\in\mathbb{Z}$.

Under this isomorphism, it was shown by Green in~\cite{Gre321} that FC elements correspond to $321$-avoiding permutations, which are the permutations $\sigma$ with no $i<j<k$ in $\mathbb{Z}$ satisfying $\sigma(i)>\sigma(j)>\sigma(k)$. This generalizes the well-known result of Billey, Jockush and Stanley from~\cite{BJS} for the case of the symmetric group.  

\subsection{Characterization}

We first use the following proposition, essentially proved in Green's paper~\cite{Gre321}.

\begin{proposition}
\label{prop:caracterisation_Atilde} An element $w$ of type $\Aaff_{n-1}$ is fully commutative if and only if, in  any reduced decomposition of $w$, the occurrences of $s_i$ and $s_{i+1}$ alternate for all $i\in \{0,\ldots,n-1\}$, where we set $s_n=s_0$.
\end{proposition}

\begin{proof}
 The condition is clearly sufficient by using Proposition~\ref{prop:caracterisation_fullycom}. To show that it is necessary, assume $w \in \Aaff^{FC}_{n-1}$ and $\mathbf{w}=s_{i_1}\cdots s_{i_l}$ is a reduced word for $w$. For the sake of contradiction, let us assume that there exist $i\in \{0,\ldots,n-1\}$ and $1\leq j<k\leq l$ such that $i_j=i_k=i$ and that for all $x$ satisfying $j<x<k$ one has $i_x\neq {i},{i+1}$. Among all possible integers $i,j,k$, pick one with $k-j$ minimal. Now consider the number $m$ of indices $x$ with $j<x<k$ and $i_x=i-1$. If $m=0$, then by successive commutations we see that the word is not reduced, which is excluded. If $m=1$, then by successive commutations one obtains a factor $s_is_{i-1}s_i$ which is excluded by Proposition~\ref{prop:caracterisation_fullycom}. If $m\geq 2$, then we have two consecutive occurrences of $s_{i-1}$ contradicting the minimality of $k-j$. This finishes the proof.
\end{proof}

This shows that FC heaps for the graph of type $\Aaff_{n-1}$ coincide exactly with alternating heaps; we will use this characterization to encode these heaps as certain paths in Section~\ref{sub:path_encoding}.

We shall represent heaps of type $\aff{A}_{n-1}$ by depicting all (alternating) chains $H_{\{s_i,s_{i+1}\}}$ for $i=0,\ldots,n-1$. To be able to represent these chains in a planar fashion, we duplicate the set of $s_0$-elements and use one copy for the depiction of the chain $H_{\{s_0,s_{1}\}}$ and one copy for $H_{\{s_{n-1},s_0\}}$. This can be seen in Figure~\ref{fig:AtildeHeap_slanted}, where the ``drawn on a cylinder" representation on the right is a deformation of the first one which makes more visible its poset structure.

\begin{figure}[!ht]
\begin{center}
 \includegraphics[width=0.6\textwidth]{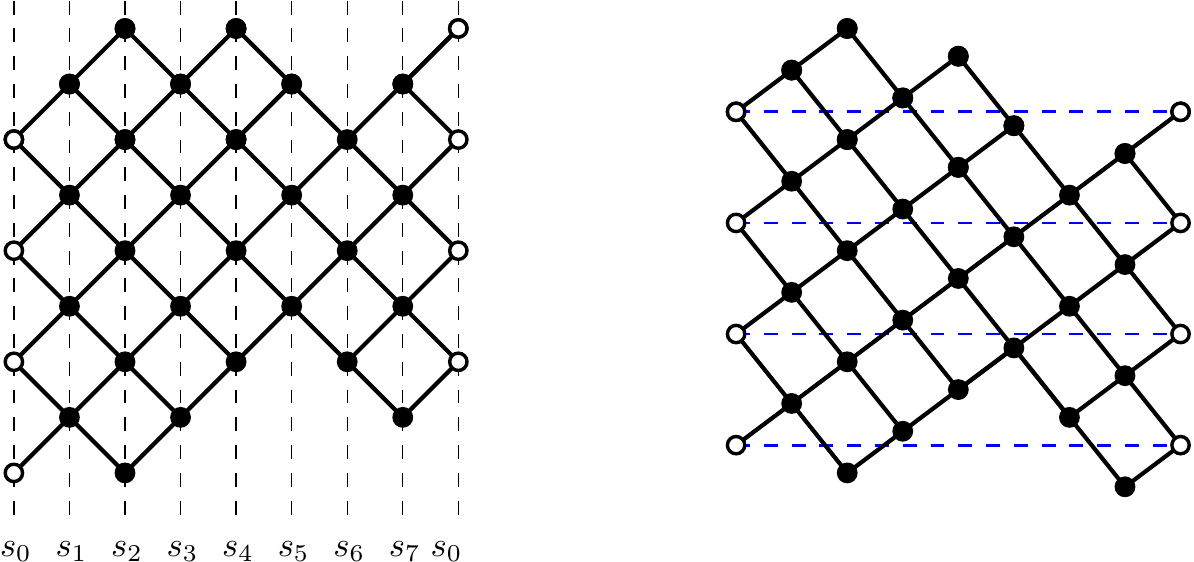}
 \caption{\label{fig:AtildeHeap_slanted} Representation of a FC heap of type $\aff{A}_7$.}
 \end{center} 
\end{figure}

\subsection{Path encoding}
\label{sub:path_encoding}

We are not exactly in the case of Section~\ref{sub:alternating} since the Coxeter graph is not linear. One can nonetheless define walks from the alternating heaps described in Proposition~\ref{prop:caracterisation_Atilde}, as follows: given a FC heap $H$ of type $\Aaff_{n-1}$ and $i=0,\ldots,n-1$,  draw a step from $P_i=(i,|H_{s_i}|)$ to $P_{i+1}=(i+1,|H_{s_{i+1}}|)$ as in Definition~\ref{def:map_phi} for $\varphi$; here we set $s_n=s_0$. This forms a path $\varphi'(H)$ of length $n$, with both $P_0$ and $P_n$ at height $|H_{s_0}|$. If $w \in \Aaff_{n-1}^{FC}$, we can set $\varphi'(w):=\varphi'(\H(w))$ since we showed that all FC heaps are alternating.

\begin{figure}[!ht]
\begin{center}
 \includegraphics[width=\textwidth]{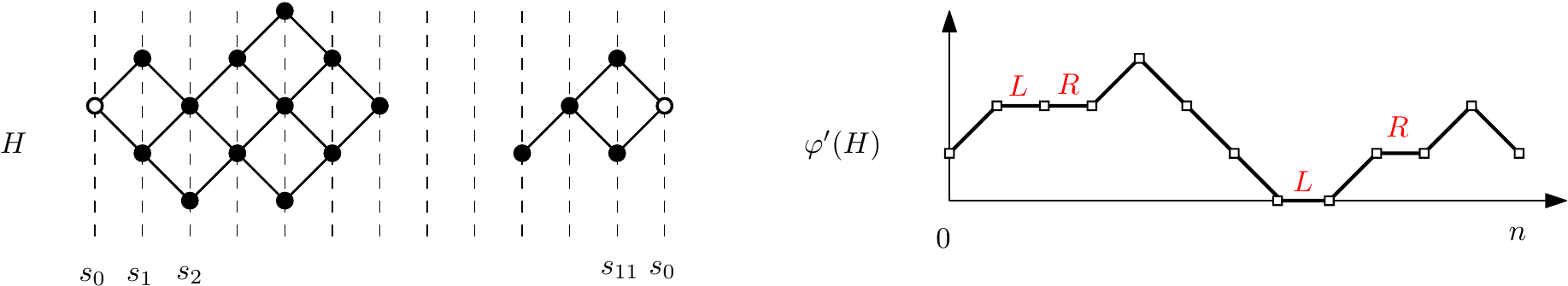}
 \end{center} 
\end{figure}

Define accordingly $\Cylset_{n}$ as the paths in $\Genset_{n}$ whose starting and ending points are at the same height. Define also $\haut'(P)=\sum_{i=0}^{n-1}\haut(P_i)$, which corresponds to the area under $P$, and let $\Cyl_n(q)$ and $\Cyl_n^*(q)$ be the generating functions with respect to $\haut'$ of $\Cylset_{n}$ and $\Cylset_{n}^*$ respectively. Finally, denote by $\E_n\subseteq \Cylset_{n}^*$ the set of walks with all vertices at the same positive height, and all $n$ steps with the same label (either $L$ or $R$).


\begin{theorem}
\label{theorem:walks_type_Atilde}
The map $\varphi':W^{FC}\rightarrow \Cylset^*_n\setminus\E_n$ is a bijection such that $\ell(w)=\haut'(\varphi'(w))$.
\end{theorem}

\begin{proof}

{\bf $\varphi'$ is well defined:} We have to show that $\varphi'(w)$ cannot belong to $\E_n$. Let $w\in \Aaff^{FC}_{n-1}$ and consider any reduced word $\mathbf{w}$ for it.  Set $j_0,j_1,\ldots,j_{n-1}$ the indices in $\mathbf{w}$ of the leftmost occurrences of the letters $s_0,s_1,\ldots,s_{n-1}$ respectively (if a $s_i$ does not occur then clearly $\varphi'(w)\notin \E_n$). Now assume $\varphi'(w)$ consists of horizontal steps all labeled $R$ (\emph{resp. $L$}): by construction, this would entail $j_0<j_1<\cdots<j_{n-1}<j_0$ (\emph{resp.} $j_0>j_1>\cdots>j_{n-1}>j_0$) which is a contradiction.

 {\bf  $\varphi'$ is {injective}:} By Theorem~\ref{theorem:walk_encoding}, the walk $\varphi'(w)$  allows us to reconstruct uniquely $H:=\H(w)$. Indeed each step $(i,h)\rightarrow(i+1,h')$ encodes the chain $H_{\{s_i,s_{i+1}\}}$, and $H$ is the transitive closure of these chains by Definition~\ref{defi:heaps}.

{\bf $\varphi'$ is surjective:} Let $P$  be a walk in $\Cylset^*_n\setminus\E_n$. Each step $(i,h)\rightarrow(i+1,h')$ of $P$ encodes a chain $C_{s_i,s_{i+1}}$, considered here as an oriented path with vertices $s_i^{(j)}$ for $j=1,\ldots, h$ and $s_{i+1}^{(j')}$ for $j'=1,\ldots, h'$. What we must show is that the oriented, labeled graph $G_P$ formed by the union of these chains is an {\em acyclic} graph, which means that it possesses no oriented cycle. In this case, $G_P$ naturally determines a labeled poset, which is an alternating heap $H$ such that $\varphi'(H)=P$.

 For each $i\in\{0,\ldots,n-1\}$, we denote by $s_i^{(j)}, j=1,\ldots,h_i$ the elements of $G_P$ labeled $s_i$ from bottom to top. By rotational symmetry, it is enough to show paths starting from a vertex $s_0^{(j)}$ cannot also end at $s_0^{(j)}$. The key property of $G_P$ is the following: if $s_i^{(j)}\rightarrow s_{i\pm 1}^ {(j')}$, then $j'=j$ or $j'=j+1$; this follows readily from the alternating condition. An immediate consequence is that the vertices $s_i^{(j)}$  involved in a possible cycle of $G_P$ necessarily have all the same index $j$. Another easy remark is that if $s_i^{(j)}\rightarrow s_{i\pm 1}^ {(j')}\rightarrow s_i^{(j'')}$, then $j''=j+1$. Therefore steps in a given cycle in $G_P$ must be all of the form $s_i^{(j)}\rightarrow s_{i+1}^ {(j)}$, or all of the form $s_i^{(j)}\rightarrow s_{i-1}^ {(j)}$.
 
 Since $P\notin \E_n$, we claim that there exists $i\in\{0,\ldots,n-1\}$ such that either $s_{i-1}^{(1)}>s_i^{(1)}<s_{i+1}^{(1)}$ or $s_{i-1}^{(1)}<s_i^{(1)}>s_{i+1}^{(1)}$: indeed,  if these bottom points were strictly increasing or decreasing in $G_P$, the same would be true for top points, and this would correspond to a path in $\E_n$.  From this we can deduce that there is no cycle from $s_0^{(1)}$ to itself. Now assume there exists a cycle from $s_0^{(j)}$ to itself for a $j>1$, and pick $j$ minimal with this property. This cycle can only involve vertices of the form $s_i^{(j)}$ as seen above, and because of the alternating property, replacing $j$ by $j-1$ creates a cycle from $s_0^{(j-1)}$ to itself. This contradicts the minimality of $j$, and completes the proof of surjectivity.
\end{proof}

\subsection{Periodicity and generating functions}

Let $a^n_\ell$ denote the number of FC elements of type $\aff{A}_{n-1}$ and of length $\ell$.

\begin{theorem}\label{theo:periodicityA}
Fix $n>0$. In type $\aff{A}_{n-1}$, the growth sequence $(a^{n}_\ell)_{\ell\geq 0}$ is ultimately periodic of period $n$. Moreover periodicity starts at length $\ell_0+1$, where  $\ell_0=\lfloor \frac{n-1}{2}\rfloor \lceil \frac{n-1}{2} \rceil$.
\end{theorem}


\begin{proof}
By Theorem~\ref{theorem:walks_type_Atilde}, $a^n_\ell$ is the number of walks in $\Cylset^*_n\setminus\E_n$ with area $\ell$. Assume $\ell$ is large enough so that paths of area $\ell$ will not have any horizontal step at height $0$. In this case $S:\Cylset^*_n(\ell)\to \Cylset^*_n(\ell+n)$ which shifts every vertex up by one unit (and preserves labels $L,R$) is a bijection.

 A simple calculation shows that $\ell_0$ is the largest area for which there exists such a path in $\Cylset^*_n$ with a horizontal step at height $0$ (see illustration below). It is easy to show that $a^{n}_{\ell_0}=a^n_{\ell_0+n}-n$ when $n$ is odd and $a^{n}_{\ell_0}=a^n_{\ell_0+n}-2n$ when $n$ is even, which completes the proof.
\end{proof}

\begin{figure}[!ht]
\begin{center}
 \includegraphics[width=0.7\textwidth]{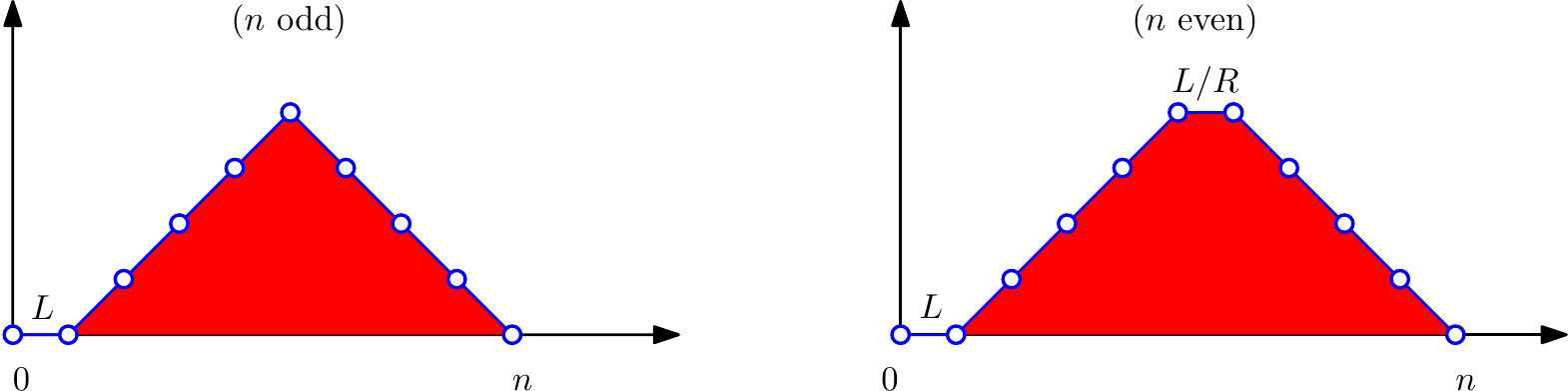}
 \end{center} 
\end{figure}

The periodicity was proved in \cite{HanJon}, while the exact beginning of periodicity was only conjectured by the authors.

\begin{corollary}
\label{cor:Atilde}
The generating function of FC elements of type $\aff{A}_{n-1}$ is
\begin{equation}
 \label{eq:Aaffine}
\aff{A}_{n-1}^{FC}(q)=\Cyl^*_n(q)-\frac{2q^n}{1-q^n}= \frac{q^{n}(\touch{\Cyl}_n(q)-2)}{1-q^{n}}+\touch{\Cyl}^*_n(q).
\end{equation}
Moreover, $\aff{A}_{n-1}^{FC}(q)$ can be computed through the equations:
\begin{equation}
\label{eq:aff_touch}
\left\{\begin{array}{l}\displaystyle\touch{\Cyl}(x)=\Motzbic(x)\left( 1+qx^2{\frac{\partial(xM)}{\partial x}}(qx) \right)\\\displaystyle\touch{\Cyl}^*(x)=\Motz(x)\left( 1+qx^2{\frac{\partial(xM)}{\partial x}}(qx) \right).\end{array}\right.
\end{equation}
Here $\Motz(x)$ and  $\Motzbic(x)$ are calculated thanks to~\eqref{relationA1} and~\eqref{eqfonctMetoile}.

\end{corollary}

\begin{proof}
By Theorem~\ref{theorem:walks_type_Atilde}, we directly have that $\Aaff_{n-1}^{FC}(q)=\Cyl^*_n(q)-{2q^n}/(1-q^n)$. Now we have to count walks in $\Cylset^*_{n}$, and to this end we decompose them according to the lowest height they reach as follows: 
\begin{center}
 \includegraphics[width=0.6\textwidth]{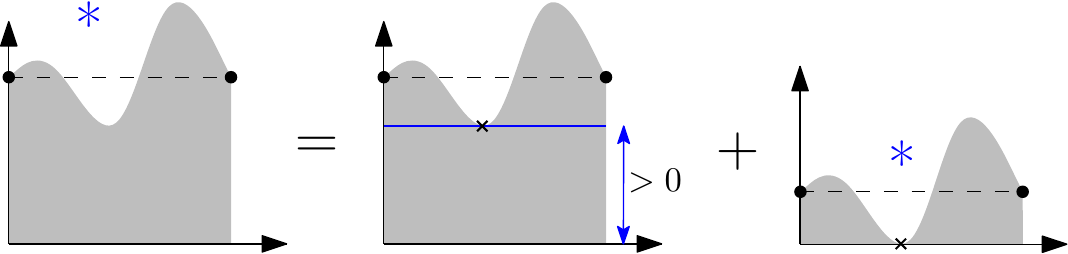}
 \end{center} 
 This gives the equation $\Cyl^*_n(q)={q^n}\touch{\Cyl}_n(q)/(1-q^n)+\touch{\Cyl}^*_n(q)$, and therefore~\eqref{eq:Aaffine}.
Finally, to compute $\touch{\Cyl}_n(q)$ and $\touch{\Cyl}^*_n(q)$, note that walks in $\touch{\Cylset}_n$ (\emph{resp.}  $\touch{\Cylset}^*_n$) either belong to $\Motzbicset_n$ (\emph{resp.} $\Motzset_n$) or can be decomposed as in the picture below: 

\begin{center}
\includegraphics[width=0.4\textwidth]{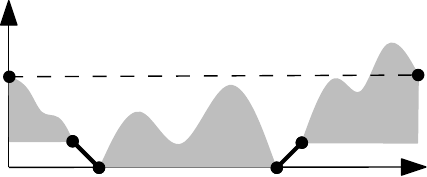}
 \end{center}
 
Here the central part is a path in  $\Motzbicset_j$ (\emph{resp.} $\Motzset_j$), for some integer $j$. The two other parts can be joined to form a path in $\Motzbicset_{k}$ where $k=n-j-2$, shifted up, and with a marked vertex to record where the parts are glued. The area generating polynomial of these marked paths is given by $q^{k+1}(k+1)M_k(q)$, whose generating  function is
 $q\frac{\partial(xM)}{\partial x}(qx)$ and the equations~\eqref{eq:aff_touch} follow.
\end{proof}

We now give the following result, whose easy proof is left to the reader.
\begin{lemma}
\label{lemma:periodic}
Suppose $F(q)=\sum_{i\geq 0}f_iq^i={P(q)}/(1-q^N)$ where $N$ is a positive integer and  $P(q)$ is a polynomial of degree $d$. Then one has $f_{i+N}=f_i$ for all $i\geq d$. Furthermore the mean value over a period $(f_i+f_{i+1}+\cdots+f_{i+N-1})/N$ for $i\geq d$ is equal to $P(1)/N$.
\end{lemma}

\begin{proposition}
\label{prop:meanA}
The mean value $\mu_{\Aaff_{n-1}}$ of $(a^{n}_\ell)_{\ell\geq 0}$ is equal to $\frac{1}{n}\left(\binom{2n}{n}-2\right).$
\end{proposition}

\begin{proof}
By Lemma~\ref{lemma:periodic} and Corollary~\ref{cor:Atilde},  we have to prove $\touch{\Cyl}_n(1)=\binom{2n}{n}$. To see this, shift any path of $\touch{\Cylset}_n$ so that it starts at the origin, and use the transformations $U\mapsto UU,D\mapsto DD,L\mapsto UD,R\mapsto DU$ defined in Remark~\ref{rem:MotzToDyck}. This is a bijection from $\touch{\Cylset}_n$ to {\em bilateral} paths, i.e. paths from the origin to $(2n,0)$ using steps $U$ or $D$. There are obviously $\binom{2n}{n}$ such paths, which concludes the proof.
\end{proof}



\subsection[Type A]{Type $A$}
\label{sub:finA}

We can deduce results in type $A$ from those in type $\aff{A}$. FC elements of type $A_{n-1}$ are in bijection with FC elements of type $\aff{A}_{n-1}$ whose reduced expressions have no occurrence of $s_0$. Therefore by Theorem~\ref{theorem:walks_type_Atilde}, FC elements of type $A_{n-1}$ are in bijection with elements of $\Cylset_n^*$ with starting and ending points at height $0$, which is precisely the set $\Motzset_n$. Since weights are preserved, we deduce the following enumerative result about $A^{FC}(x):=\sum_{n\geq1}x^nA_{n-1}^{FC}(q)$, which was obtained in a different way by Barcucci et al. in~\cite{BDPR}.

\begin{proposition}\label{typeA}
We have $\displaystyle A_{n-1}^{FC}(q)=\Motz_{n}(q)$, and equivalently 
$$\displaystyle{A^{FC}(x)=\Motz(x)-1}.$$
\end{proposition}

In particular, since $\Motzset_n$ is in bijection with Dyck walks by Remark~\ref{rem:MotzToDyck}, this shows that ${A}^{FC}_{n-1}$ has cardinality the $n$th Catalan number $\frac{1}{n+1}\binom{2n}{n}$: this is a well-known result, which is proved for example in~\cite{BJS,Fan,St3}.

\begin{corollary}\label{italiens}
The generating function $A^{FC}(x)$ satisfies the following functional equation
\begin{equation}\label{eqfonctA}
A^{FC}(x)=x+xA^{FC}(x)+qxA^{FC}(x)(A^{FC}(qx)+1).
\end{equation}
\end{corollary}

\begin{proof} We have the following walk decompositions with corresponding equations:

\begin{center}
\includegraphics[width=0.7\textwidth]{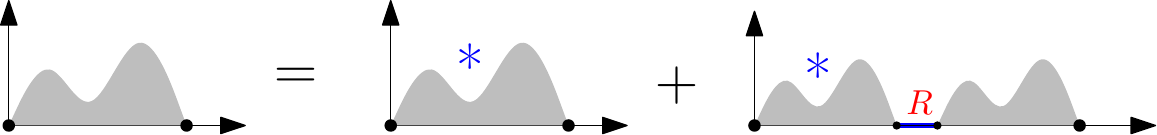}
\end{center}

\begin{equation}\label{relationA1}
\Motzbic(x)=\Motz(x)+x\Motz(x)\Motzbic(x),
\end{equation}

\begin{center}
\includegraphics[width=0.7\textwidth]{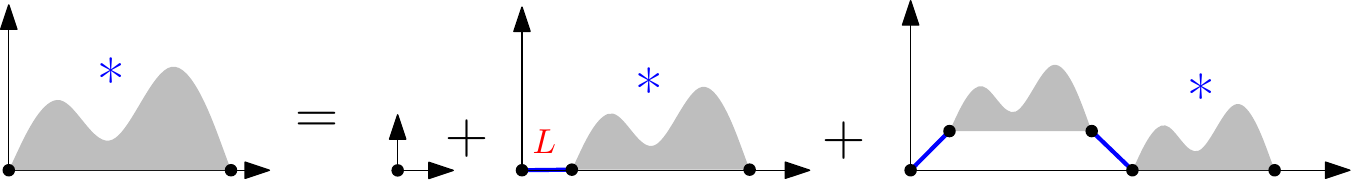}
\end{center}

\begin{equation}\label{relationA2}
\Motz(x)=1+x\Motz(x)+qx^2\Motz(x)\Motzbic(qx).
\end{equation}

\bigskip

 The identity~\eqref{relationA1} gives $\Motzbic(x)=\Motz(x)/(1-x\Motz(x))$, which can be replaced in~\eqref{relationA2} to yield after some simplifications:
\begin{equation}\label{eqfonctMetoile}
\Motz(x)=1+x\Motz(x)+qx(\Motz(x)-1)\Motz(qx).
\end{equation}
Finally, we see through Proposition~\ref{typeA} that~\eqref{eqfonctMetoile} is exactly Equation~\eqref{eqfonctA}.
\end{proof}

Corollary~\ref{italiens} gives another proof of~\cite[Eq. (3.0.2)]{BDPR}, where $A^{FC}(x)$ is denoted by $C(x,q)$. In their work, Barcucci et al. also proved an expression for $C(x,q)$ as a quotient of $q$-Bessel type functions, using a recursive rewriting rule for $321$-avoiding permutations and a result of  Bousquet-M\'elou~\cite[Lemma~2.2]{Mireille}. It is possible to derive their ratio of $q$-Bessel functions by writing 
\begin{equation}\label{qbessel}
A^{FC}(x)+1=\Motz(x)=\frac{\sum_{n\geq0}\alpha_n(q)x^n/(x;q)_{n+1}}{\sum_{n\geq0}\alpha_n(q)x^n/(x;q)_{n}},
\end{equation}
where $(x;q)_n=(1-x)\cdots(1-xq^{n-1})$ stands for the classical $q$-rising factorial, and $\alpha_n(q)$ is a $q$-hypergeometric coefficient. Then plugging~\eqref{qbessel} into~\eqref{eqfonctMetoile} yields $\alpha_n(q)=\alpha_n(q)q^{-n}+\alpha_{n-1}(q)$, from which we deduce 
$$\alpha_n(q)=-\frac{q^n}{1-q^n}\alpha_{n-1}(q),$$
and finally $\alpha_n(q)=(-1)^nq^{n(n+1)/2}\alpha_0(q)/(q;q)_n$, as in~\cite{BDPR}. 

\subsection{Connection to previous related works}

There are many papers which investigate fully commutative elements in type $\aff{A}$, and in some of them one can find explicit characterizations of these elements. We will give a brief account of the four works related to the present one we could find in the literature, and explain their relationship with ours. We also believe this is all the more a worthy enterprise since, based on their respective citations, none of these works seemed to be aware of the earlier ones.

The first work~\cite{FanGreen_Affine} is due to Fan and Green and does not explicitly set out to characterize fully commutative elements. Its goal is to study the Temperley--Lieb algebra of type $\aff{A}$, which has a basis indexed by FC elements. The authors then give a diagrammatic representation of this algebra, by sending the aforementioned basis to an explicit basis of admissible diagrams. This last basis being rather explicit, it gives in return some characterization of FC elements; see Section~\ref{sub:diagramTL} for more informations. 

The second article~\cite{HagiwaraAtilde} of Hagiwara deals with \emph{minuscule heaps of type $\aff{A}$}. In general, minuscule heaps are a strict subset of FC heaps, but here they coincide, as the author shows in his Theorem 5.1. His characterization goes by embedding posets in a family of slanted lattices $L_k$, and Hagiwara proves that FC heaps are precisely the finite convex subsets occurring in a lattice $L_k$. It can be easily seen that this is a corollary of our work; the gradient $k$ defined in~\cite[p.17]{HagiwaraAtilde} can be seen in the path $\varphi'(H)$ as the sum of the number of up steps and the number of horizontal $R$ steps.

The paper~\cite{HanJon} by Hanusa and Jones was already mentioned several times. Here the FC permutations are classified and counted by dividing them first into long and short ones. Long permutations are easily counted and have a pleasing generating function, while the enumeration of short ones requires several pages resulting in a rather complicated generating function. As mentioned before, we could confirm their conjecture about the precise beginning of the periodicity. One can see \emph{a posteriori} that we could manage this by considering all elements in our approach, without dividing them beforehand into adequate ``long'' and ``short'' ones.

In the recent work~\cite{AlHarbat} by Al Harbat, the author classifies FC elements by indicating a normal form for each of them. That is, the main theorem exhibits a family of reduced FC expressions where each FC element is represented exactly once. We will not detail this here, but these normal forms correspond to a particular linear extension of FC heaps which is fairly easy to describe.

\section[Classical affine types]{Classical affine types $\widetilde{B},\widetilde{C},\widetilde{D}$
}
\label{sec:bcd}

In this section we classify FC elements in types $\aff{C}_n, \aff{B}_{n+1}$, and $\aff{D}_{n+2}$ thanks to their heaps. As could be perhaps expected, the problem is much subtler than in  type $\aff{A}_{n-1}$ and in particular several kinds of elements appear. The pleasant part about our point of view is that we are able to formulate our proof so that the same one essentially works for all three cases.

In each case we will show that the growth sequence of FC elements is ultimately periodic, and we will be able to compute the generating function. Finally, we can easily specialize our characterizations to finite types $B_n$ and $D_{n+1}$, and obtain enumerative results.

\begin{center}
\includegraphics{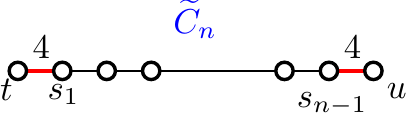}
\end{center}
\subsection[Fully commutative heaps]{Fully commutative heaps of type $\aff{C}_n$}
\label{sub:Caffine}

We need a couple of notations. A {\em peak} is a heap of the form $P_{\rightarrow}(s_i):=\H (s_i s_{i+1}\dots s_{n-1}u s_{n-1}\dots s_{i+1}s_i)$ or $P_{\leftarrow}(s_i):=\H (s_i s_{i-1}\dots s_{1}t s_{1}\dots s_{i-1}s_i)$. If $H$ is a heap  of type $\aff{C}_n$ and $i\in\{1,\ldots,n-1\}$, then $H_{\{\leftarrow s_i\}}$ ({\em resp.} $H_{\{\rightarrow s_i\}}$) denotes the restriction of $H$ to the labels $\{t,s_1,s_2\ldots,s_i\}$ ({\em resp.} $\{s_i,s_{i+1},\ldots,s_{n-1},u\}$).

\begin{definition}\label{def:famillesCtilde}  We define the five following families of heaps of type $\aff{C}_n$.

\noindent \textbf{(ALT)} {\em Alternating heaps}. $H\in $ (ALT) if it is alternating in the sense of Definition~\ref{defi:alternating} where $\Gamma$ is the Coxeter diagram $\aff{C}_n$.

\noindent \textbf{(ZZ)} {\em Zigzags}. $H\in $ (ZZ) if $H=\H(\mathbf{w})$ where $\mathbf{w}$ is a finite factor of the infinite word $\left(ts_1s_2\cdots s_{n-1}us_{n-1}\cdots s_2s_1\right)^\infty$ such that $|H_{s_i}| \geq 3$ for at least one $i\in\{1,\dots,n-1\}$.

 \noindent \textbf{(LP)}  {\em Left-Peak}. $H\in $(LP) if there exists $j \in\{1,\dots,n-1\}$ such that:
\begin{enumerate}
\item  $H_{\{\leftarrow s_j\}}=P_{\leftarrow}(s_j)$;
\item If $j\neq n-1$ then there is no $s_{j+1}$-element between the two $s_j$-elements; if $j=n-1$ then there is no $u$-element between the two $s_{n-1}$-elements;
\item  $H_{\{{s}_{j}\rightarrow \}}$ is alternating when one $s_j$-element is deleted.
\end{enumerate}
\noindent \textbf{(RP)} {\em Right-Peak}. $H\in $ (RP) if there exists $k \in\{1,\dots,n-1\}$ such that:
\begin{enumerate}
\item $H_{\{s_k \rightarrow\}}= P_{\rightarrow}(s_k)$;
\item If $k\neq 1$ then there is no $s_{k-1}$-element between the two $s_k$-elements; if $k=1$ then there is no $t$-element between the two $s_1$-elements. 
\item $H_{\{\leftarrow {s}_{k}\}}$ is alternating when one $s_k$-element is deleted.
\end{enumerate}
\noindent \textbf{(LRP)} {\em Left-Right-Peak}. $H\in $ (LRP) if there exist $1\leq j<k \leq n-1$ such that:
\begin{enumerate}
\item  LP(1), LP(2), RP(1), RP(2) hold;
\item  $H_{\{s_{j},\ldots,s_{k}\}}$ is  alternating when both a $s_j$- and a $s_k$-element are deleted.
\end{enumerate}
\end{definition}

\begin{figure}[t]
\includegraphics[width=0.98 \textwidth]{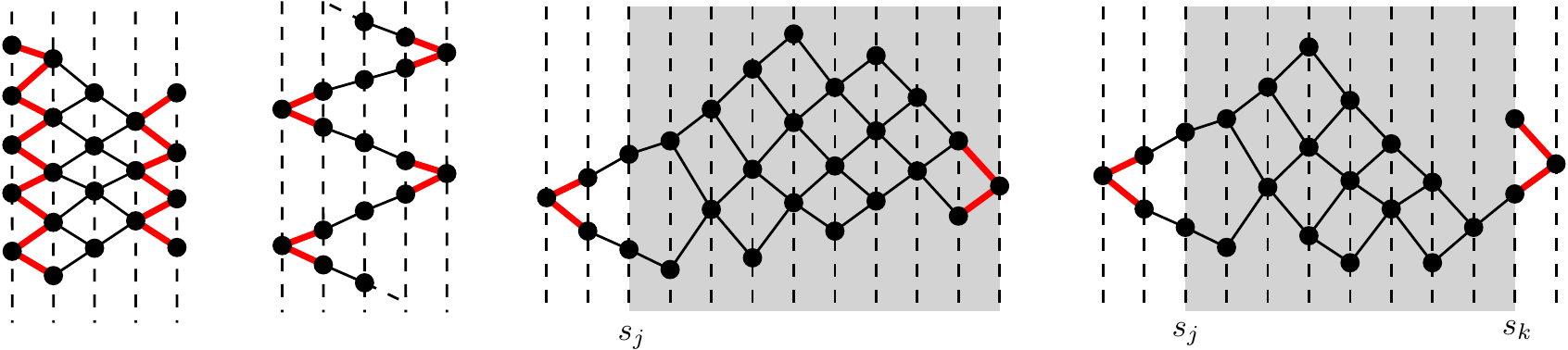}
\caption{Examples  of heaps of type  $\aff{C}_n$ in families (ALT), (ZZ), (LP), (LRP).}
\end{figure}

Before stating the main theorem, we make a couple of important remarks on these definitions.

\begin{remark}
\label{rem:familyZZ}
The condition $|H_{s_i}|\geq 3$ in the definition of (ZZ) is only there to ensure that the families are disjoint, as we will see in Proposition~\ref{prop:sufficient}. 
\end{remark}

\begin{remark}
\label{rem:familyPeaks}
In families (LP), (RP), (LRP), the indices $j$ and $k$ are {\em uniquely determined}; this will be particularly useful for enumerating purposes. 

 The \emph{extremal cases} for these families are the heaps having no chain $H_{\{s_i,s_{i+1}\}}$ which alternates for $i\in\{1,\ldots,n-2\}$. They correspond to $j=n-2$ or $n-1$ in the case of family (LP); $k=1$ or $2$ for the family (RP); and $k-j=1$ or $2$ in the case of family (LRP). These possibilities are illustrated in Figure~\ref{fig:extremal_heaps}, up to obvious horizontal or vertical symmetries.
\end{remark}

We can now state the main theorem of this section.

\begin{theorem}[Classification of FC heaps of type $\aff{C}_n$]\label{theo:affineCfamilles}
A heap of type $\aff{C}_n$ is  fully commutative if and only if it belongs to one of the five families (ALT), (ZZ), (LP), (RP), (LRP).
\end{theorem}
 
We start with the easy direction, by showing that these families contain indeed only fully commutative elements.

\begin{proposition}\label{prop:sufficient} The families (ALT), (ZZ), (LP), (RP), (LRP) are pairwise disjoint, and contain only fully commutative heaps of type $\aff{C}_n$.
\end{proposition}

\begin{proof}
Let $H$ be a heap of type $\aff{C}_n$ belonging to one of the five families. If $H\in$ (ALT), we have by Proposition~\ref{prop:altFCheaps} that it is a FC heap. Now condition $(b)$ of Proposition~\ref{prop:heaps_fullycom} is easily verified by inspection for all families, so we must verify condition $(a)$. This is clear for (ZZ). For the three remaining families, a bad convex chain neither  occurs inside the peak parts, nor inside the alternating parts. The only chains to verify are those occurring at the junction of these parts, and  concern labels $s_{j}$ and $s_{j+1}$ (families (LP) and (LRP)) and labels $s_{k}$ and $s_{k-1}$ (families (RP) and (LRP)). It is an easy verification that no bad convex chain occurs, which concludes the proof that all heaps are indeed FC.

To see that the families are disjoint, we notice that (ALT) is the only family containing alternating heaps. Also (ZZ) is the only family with no alternating chain $H_{\{s_i,s_{i+1}\}}$ and at least three $s_{i_0}$-element for a certain $s_{i_0}$: indeed, if $H$ is in (LP), (RP) or (LRP), it has no alternating chain only in the extreme cases singled out in Remark~\ref{rem:familyPeaks}, the corresponding possible heaps being illustrated in Figure~\ref{fig:extremal_heaps}. By inspection, one notes that for any $i\in\{1,\ldots,n-1\}, |H_{s_i}| \leq 2$. Finally, the families (LP), (RP) or (LRP) are pairwise disjoint since they are easily distinguished by comparing their peaks.
\end{proof}

We now collect three lemmas, which will be useful in the proof of Theorem~\ref{theo:affineCfamilles}.

\begin{lemma}\label{lemma:forbiddenC}
  Let $H$ be a FC heap of type $\aff{C}_n$, and $i\in \{2,\dots,n-1\}$.  Then in $H_{\{s_{i-1},s_{i}\}}$ there is no factor equal to: 
\begin{center} (1) $s_{i}s_{i-1}s_{i}s_{i}$,  (1') $s_{i}s_{i}s_{i-1}s_{i}$, or (2) $s_{i}s_{i}s_{i}$. \end{center}
\end{lemma}

\begin{lemma}[Peak Lemma]\label{lemma:peakC}
Suppose that there exists $i\in \{2,\ldots, n-1\}$ such that $H_{\{s_{i-1},s_{i}\}}$ contains the factor $s_{i}s_{i}$. Then the interval of $H$ between the two corresponding $s_{i}$-elements is isomorphic to a peak $P_{\rightarrow}(s_i)$.
\end{lemma}

\begin{lemma}[Zigzag Lemma]
\label{lemma:zigzagC}
Let $H$ be a FC heap of type $\aff{C}_n$ having two  $s_i$-elements delimitating an interval of the form $P_{\rightarrow}(s_i)$ or $P_{\leftarrow}(s_i)$ for some $i\in \{1,\dots,n-1\}$. If $H$ contains at least another $s_i$-element, then $H \in$ (ZZ).
\end{lemma}

These three lemmas will be proved in Section~\ref{sub:proofs} in a more general context. Note that Lemma~\ref{lemma:zigzagC} is equivalent to~\cite[Lemma 3.3.6]{ErnstDiagramII}.

\begin{figure}[!t]
\includegraphics[width=\textwidth]{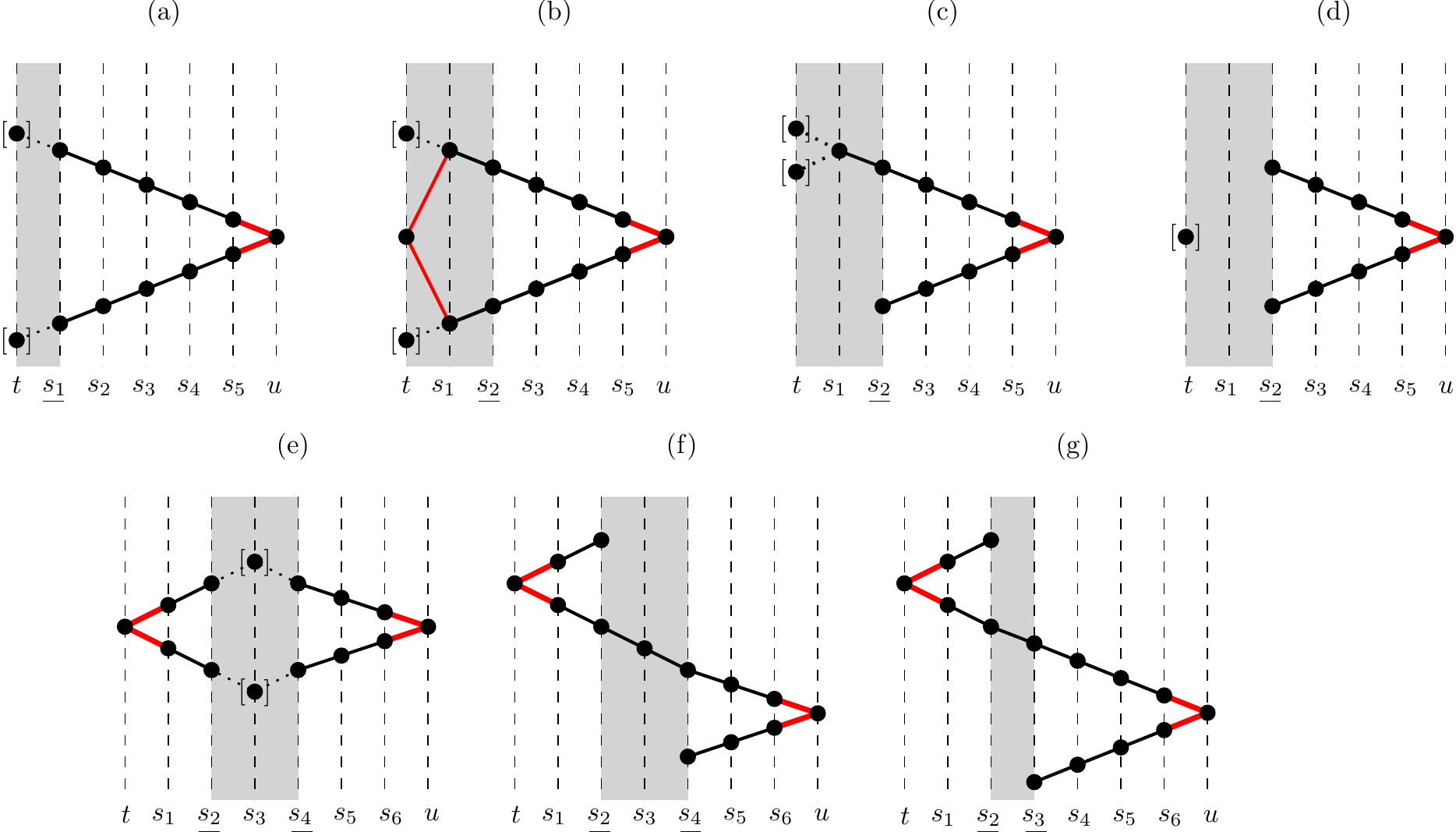}
\caption{Exceptional heaps of type (RP) in the first row, and (LRP) in the second row. \label{fig:extremal_heaps}}
\end{figure}

\begin{proof}[Proof of Theorem~\ref{theo:affineCfamilles}]
By Proposition~\ref{prop:sufficient}, it remains to show that all FC heaps belong to one of the five families of Definition~\ref{def:famillesCtilde}. So let $H$ be a FC heap of type $\aff{C}_n$, and consider the two following properties:
\begin{itemize}
\item[$(P_1)$] There exists $i_0\in \{1,\ldots,n-2\}$ such that $H_{\{s_{i_0},s_{i_0+1}\}}$ is alternating;
\item[$(P_2)$] There exists $i_1\in \{1,\ldots,n-1\}$ such that $|H_{s_{i_1}}|\geq 3$.
\end{itemize}

\noindent $\bullet$ Assume first that $H$ satisfies $(P_1)$. Denote by $j$ the smallest such $i_0$, and let $k \leq n-1$ be the largest index such that 
$H_{\{s_{j} ,\ldots, s_{k}\}}$ is alternating.

If $k=n-1$, so that $H_{\{s_{n-2},s_{n-1}\}}$ is alternating, then the full commutativity of $H$ implies that $H_{\{s_{n-1},u\}}$ is also alternating, and thus we have that $H_{\{s_j\rightarrow\}}$ is alternating.

If $k<n-1$, then  $H_{\{s_{k-1}, s_k\}}$ is alternating, but $H_{\{s_k,  s_{k+1}\}}$ is not. Since $H_{\{s_{k-1}, s_k\}}$
is alternating there exists at least one $s_{k+1}$-element between two $s_k$-elements; and since $H_{\{s_k,  s_{k+1}\}}$ is not alternating there exist two $s_{k+1}$-elements with no $s_k$-element between them. By Lemma~\ref{lemma:forbiddenC}  this is only possible if $|H_{s_{k}}|\leq 2$ and $|H_{s_{k+1}}|=2$. Hence by Lemma~\ref{lemma:peakC}, $H$ verifies conditions (RP)(1) and (2).

Symmetrically, $j=1$ implies $H_{\{\leftarrow s_k\}}$ alternating while $j>1$ implies that $H$ verifies conditions (LP)(1) and (2).
Putting things together, we obtain $H\in {\rm (ALT)}$ if $j=1$ and $k=n-1$; $H\in$ (RP) if $j=1$ and $k<n-1$; $H\in$ (LP) if $j>1$ and $k=n-1$; $H\in$ (LRP) if $j>1$ and $k<n-1$. 

\noindent $\bullet$ Assume now that $(P_2)$ holds but not $(P_1)$. If $i_1>1$, using the fact that $H_{\{s_{i_1},s_{i_1+1}\}}$ is {not} alternating, and the forbidden configurations of Lemma~\ref{lemma:forbiddenC}, there exist necessarily two $s_{i_1}$-elements with no $s_{i_1-1}$-element between them. By Lemma~\ref{lemma:peakC} this implies that the interval between the two $s_{i_1}$-elements is a peak $P_{\rightarrow}(s_{i_1})$, and by Lemma~\ref{lemma:zigzagC} we can conclude that $H \in$ (ZZ). 

\noindent $\bullet$ Assume finally that $H$ satisfies neither $(P_1)$ nor $(P_2)$, so that no chain $H_{\{s_i,s_{i+1}\}}$ is alternating and no $s_i$-element occurs more than twice. There is an $i_0\in\{1,\ldots,n-1\}$ such that $|H_{s_{i_0}}|=2$, otherwise the heap would be alternating. Suppose $i_0<n-1$. Since $H_{\{s_{i_0},s_{i_0+1}\}}$ is not alternating, there is either $0$ or $2$ $s_{i_0+1}$-elements between the two $s_{i_0}$-elements. In the first case we have a peak $P_{\leftarrow}(s_{i_0})$ and in the second we have a peak $P_{\rightarrow}(s_{i_0})$, by applying Lemma~\ref{lemma:peakC}. If $i_0=n-1$, we can apply the same reasoning with $i_0-1$.

So without loss of generality, we have the existence of a peak $P_{\rightarrow}(s_{i_0})$; we choose $i_0$  minimal with this property. It is now easy though a little tedious to identify all possible heaps; we just give the end result, using  the cases (a)-(g) of Figure~\ref{fig:extremal_heaps}. If $i_0=1$, the possibilities are cases (a) and (b). If $i_0>1$, let $m=|H_{s_{i_0-1}}|\in\{0,1,2\}$. If $m=0$, the possibilities are (d) if $i_0=2$, and (e) otherwise. If $m=1$, the possibilities are (c) if $i_0=2$, and (e) or (f) otherwise. Finally for $m=2$ the possibilities are (e) or (g). So we always obtain  heaps belonging to (LP), (RP) or (LRP), which achieves the proof.
\end{proof} 

\subsection{Fully Commutative elements of types  $\aff{B}_{n+1}$ and $\aff{D}_{n+2}$}
\label{sub:FCaffBD}

We will define some operations of substitution, which we use to characterize FC heaps in types $\aff{B}_{n+1}$ and $\aff{D}_{n+2}$ starting from FC heaps in type $\aff{C}_n$. The result can be summarized as follows: a FC heap of type $\aff{B}_{n+1}$  always comes from a FC heap of type $\aff{C}_n$ in which $t$-elements were replaced by elements labeled $t_1,t_2$ or $t_1t_2$, while a FC heap of type $\aff{D}_{n+2}$  always comes from a FC heap of type $\aff{C}_n$ in which additionally $u$-elements were replaced by elements labeled $u_1,u_2$ or $u_1u_2$. We now make these replacements precise.

\begin{center}
\includegraphics{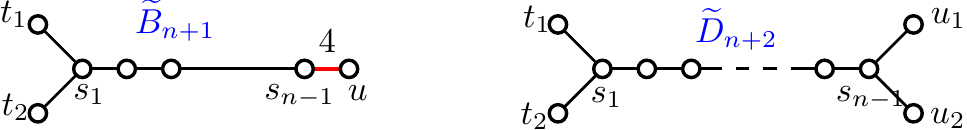}
\end{center}

{\noindent \bf Substitutions $\Delta_t$ and $\Delta_{t,u}$:} Let $H$ be a FC heap of type $\aff{C}_n$, and we fix a word $\mathbf{w}$ such that $H=\H(\mathbf{w})$ (the operations do not depend on the word $\mathbf{w}$ which is chosen). We define a set $\Delta_t(H)$ which is a collection of heaps of type $\aff{B}_{n+1}$ deduced from $H$, as follows:

$\bullet$ $H$ is in (LP) or (LRP). There is a unique $t$-element in $H$ coming from a peak $P_{\leftarrow}(s_i)$. In this case $\Delta_t(H):=\{\H(\mathbf{w'})\}$ where $\mathbf{w'}$ is obtained by replacing $t$ by $t_1t_2$ in $\mathbf{w}$.

$\bullet$ $H$ is in (ZZ). It is a chain so $\mathbf{w}$ is uniquely determined. Replace once again $t$ by $t_1t_2$ in $\mathbf{w}$ for all $t$; in the case where  $\mathbf{w}$ starts or ends with $t$, these occurrences can be replaced also by $t_1$ or $t_2 $. Then $\Delta_t(H)$ is the collection of the possible replacements: it has $1$ element if $\mathbf{w}$ neither starts nor ends with $t$, $3$ elements if $\mathbf{w}$ starts or ends with $t$ but not both, and $9$ elements if $\mathbf{w}$ both starts and ends with $t$.

$\bullet$ $H$ is in (ALT) or (RP). If $H$ contains no $t$-element, define $\Delta_t(H):=\{H\}$. If $H$ contains exactly one $t$-element,  $\Delta_t(H):=\{\H(\mathbf{w^1}),\H(\mathbf{w^2}),\H(\mathbf{w^3})\}$ where $\mathbf{w^1}$ (\emph{resp.} $\mathbf{w^2}$, \emph{resp.} $\mathbf{w^3}$) is obtained by replacing $t$ by $t_1t_2$ (\emph{resp.} $t_1$, \emph{resp.} $t_2$) in  $\mathbf{w}$. If $H$ contains more than one $t$-element, then $\Delta_t(H):=\{\H(\mathbf{w^1}),\H(\mathbf{w^2})\}$, where $\mathbf{w^1}$ (\emph{resp.} $\mathbf{w^2}$) is obtained by replacing the occurrences of $t$ from first to last alternatively by $t_1$ and $t_2$, starting with $t_1$  (\emph{resp.}  $t_2$).

There is one special case, $\H(ts_1\ldots s_{n-1}us_{n-1}\cdots s_1t)$ (cf. Figure~\ref{fig:extremal_heaps}(a)), in which $t$-elements can be independently replaced by $t_1$, $t_2$ or $t_1t_2$.

We also define $\Delta_{t,u}(H)$ which is a set of heaps of type $\aff{D}_{n+2}$ where the substitutions above are also performed similarly on $u$-elements.

\begin{definition}\label{def:famillesBDtilde}
If $X\in\{\text{(ALT),(ZZ),(LP),(RP),(LRP)}\}$ is one of the five families of FC heaps of type $\aff{C}_n$ from Definition~\ref{def:famillesCtilde}, we define the corresponding family $X_{\aff{B}}$ (\emph{resp.} $X_{\aff{D}}$) as family $\displaystyle\cup_{H\in X}\Delta_t(H)$ (\emph{resp.} $\displaystyle\cup_{H\in X}\Delta_{t,u}(H)$).
\end{definition}

The families are disjoint, since one can easily recover $H$ from $\Delta_t(H)$ or $\Delta_{t,u}(H)$, and then use Proposition~\ref{prop:sufficient}. 

\begin{theorem}\label{theo:affineBDfamilles}
A  heap of type $\aff{B}_{n+1}$ (\emph{resp.} $\aff{D}_{n+2}$) is fully commutative if and only if it belongs to one of the families of Definition~\ref{def:famillesBDtilde}.
\end{theorem}

\begin{proof}

 It is a simple verification to show that the substitution $\Delta_t$ do not create any bad convex chains $t_1s_1t_1,t_2s_1t_2,s_1t_1s_1$ or $s_1t_2s_1$, and similarly for chains involving $u_1$ or $u_2$. For other types of chains, we deduce easily from Proposition~\ref{prop:sufficient} that they also satisfy Proposition~\ref{prop:heaps_fullycom}. Therefore heaps from the families of Definition~\ref{def:famillesBDtilde} are all FC.
 
 Now let $H$ be a FC heap of type $\aff{B}_{n+1}$ or $\aff{D}_{n+2}$. By the results of Section~\ref{sub:proofs}, Lemma~\ref{lemma:forbiddenC} still holds. Now in type $\aff{B}_{n+1}$, change $t$ into $t_1t_2$ in the definition of $P_{\leftarrow}(s_i)$, and in type $\aff{D}_{n+2}$, change  in addition $u$ into $u_1u_2$ in the definition of $P_{\rightarrow}(s_i)$. Then Lemmas~\ref{lemma:peakC} and~\ref{lemma:zigzagC} still hold if (ZZ) is replaced by (ZZ)$_{\aff{B}}$ or (ZZ)$_{\aff{D}}$ in the conclusion of Lemma~\ref{lemma:zigzagC}; they are indeed specializations of Lemmas~\ref{lemma:forbidden} and ~\ref{lemma:peak}.

 Now the proof of Theorem~\ref{theo:affineBDfamilles} in type $\aff{C}_n$ rests almost uniquely on these lemmas. There is only one exception, which is the fact that if  $H_{\{s_1,s_2\}}$ alternating, then $H_{\{t,s_1\}}$ is alternating. The corresponding statement in type $\aff{B}_{n+1}$ or $\aff{D}_{n+2}$, which is easily verified, is the following: under the same hypothesis that $H_{\{s_1,s_2\}}$ is alternating, then $H_{\{t_1,t_2,s_1\}}$ is obtained from an alternating $H_{\{t,s_1\}}$ by the substitution $\Delta_t$ in the case (ALT) described above. 
 
In conclusion the proof of Theorem~\ref{theo:affineCfamilles} can be mimicked almost verbatim here. By inspection, the end result is that the heap $H$ is necessarily a member of one of the families of Definition~\ref{def:famillesBDtilde}, and the proof is complete.
 \end{proof}

\subsection{Technical results}
\label{sub:proofs}

We give here the proofs of the  lemmas used in the previous subsections. We will actually prove them in a more general context, and we will use for this the following additional lemma, which might be useful to investigate FC heaps for more general Coxeter graphs.

\begin{lemma}
\label{lemma:local}
Consider a Coxeter graph $\Gamma$ possessing $3$ vertices $v_1,v_{2},v_{3}$ with $m_{v_1v_2}=m_{v_2v_3}=3$ and $m_{v_2x}=2$ for $x\neq v_{1},v_{3}$. Let $H$ be a FC heap of type $\Gamma$, and assume $H_{\{v_1,v_2\}}$ contains a factor $\mathbf{w}$ where (a) the $v_1$-elements are never consecutive and (b) there exist two $v_2$-elements which are consecutive. Consider then the factor $\mathbf{w}'$ of the word $H_{\{v_2,v_3\}}$ induced by the $v_2$-elements of $\mathbf{w}$. Then Properties (a) and (b), with $v_1,v_2$ being replaced by $v_2,v_3$ respectively, hold in $\mathbf{w}'$.
\end{lemma}

\begin{proof} Between two $v_2$-elements $v_2^{(j)},v_2^{(j+1)}$ occurring in $\mathbf{w}$, there is by hypothesis zero or one $v_1$-element; by Proposition~\ref{prop:heaps_fullycom}, there must then be at least two $v_3$-elements in the first case and at least one in the second case. Therefore $\mathbf{w}'$ satisfies both (a) and (b). 
\end{proof}

\begin{figure}[h]
\includegraphics[width=0.88 \textwidth]{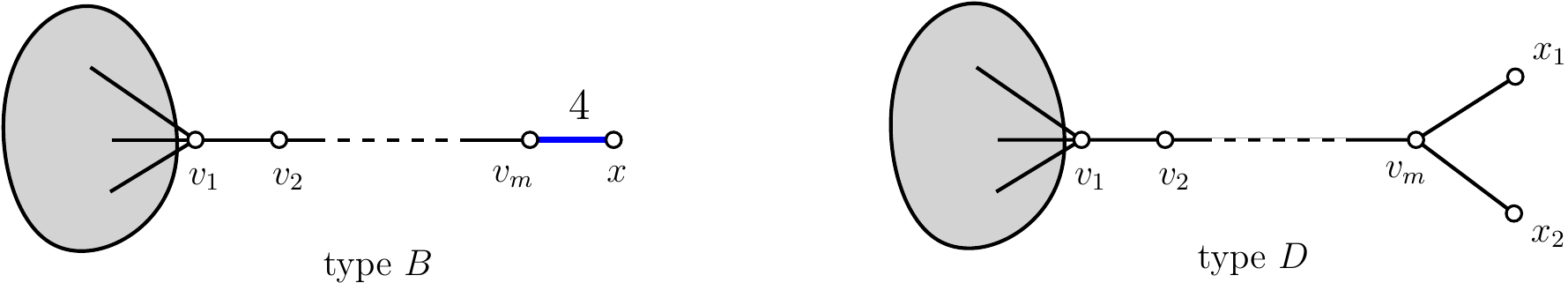}
\caption{\label{fig:branch-typeBD} Branches of types $B$ and $D$.}
\end{figure}

Now we prove the following lemma, which obviously implies Lemma~\ref{lemma:forbiddenC}.
\begin{lemma}
\label{lemma:forbidden}
Let $\Gamma$ be a Coxeter graph containing a branch of type $B$ or $D$ (see Figure \ref{fig:branch-typeBD}). Let $H$ be a FC heap of type $\Gamma$. Then for any $i\in \{2,\dots,m\}$, the word  $H_{\{v_{i-1},v_{i}\}}$ 
contains no factor equal to: \begin{center} (1) $v_{i}v_{i-1}v_{i}v_{i}$,  (1') $v_{i}v_{i}v_{i-1}v_{i}$, or (2) $v_{i}v_{i}v_{i}$. \end{center}
\end{lemma}

\noindent A graphical illustration of the forbidden factors is given in Figure~\ref{fig:forbidden}.

\begin{proof} Suppose $i=m$ and the branch is of type $B$. It is easy to see that in $H_{\{v_{m-1},v_{m}\}}$ there are no factors of type (1), (1'), or (2). Indeed, since $H$ is FC, by Proposition~\ref{prop:heaps_fullycom} there should be a $x$-element between any two consecutive $v_{m}$-elements. This would give rise to convex chains labeled either $x v_{m} x v_{m}$ or $v_{m} x v_{m} x$, contradicting  Proposition~\ref{prop:heaps_fullycom}. Similarly, in the type $D$ case, any two consecutive $v_{m}$-elements must be separated either by an occurrence of a $x_1$- or a $x_2$-element, or both. All possibilities are again excluded by Proposition~\ref{prop:heaps_fullycom}.

Now if $i<m$ and the factor is of type (1) or (1'), an immediate application of Lemma~\ref{lemma:local} shows that such a factor also occurs in $H_{\{v_{i},v_{i+1}\}}$, and we are done by induction. If the factor is of type (2), then $H_{\{v_{i},v_{i+1}\}}$ contains necessarily a factor $v_{i+1}v_{i+1}v_{i}v_{i+1}v_{i+1}$, which is impossible since it includes factors of type (1) and (1').
\end{proof}

The end of the proof also applies easily to exclude the occurrence of three $v_1$-elements with no element from $\Gamma-\{v_2,\ldots,v_m,x\}$ between them.

The next result is a generalization of Lemma~\ref{lemma:peakC}.
\begin{lemma}\label{lemma:peak}
Let $\Gamma$ and $H$ be as in Lemma~\ref{lemma:forbidden}. Suppose that there exists $i\in \{2,\ldots, m\}$ such that $H_{\{v_{i-1},v_{i}\}}$ contains a factor $v_{i}v_{i}$. Then the interval of $H$ between these two $v_{i}$-elements is isomorphic to $\H (v_i v_{i+1}\dots v_{m}x v_{m}\dots v_{i+1}v_i)$ for a branch of type $B$, and $\H (v_i v_{i+1}\dots v_{m}x_1x_2 v_{m}\dots v_{i+1}v_i)$ for a branch of type $D$.
\end{lemma}

\begin{proof}
Suppose  $i=m$, and  the factor $v_{m}v_{m}$ is in $H_{\{v_{m-1},v_{m}\}}$. Since $H$ is FC, Proposition~\ref{prop:heaps_fullycom}(b) implies that exactly one element ({\em resp.} two elements)  labeled $x$ ({\em resp.} $x_1$ and $x_2$) must occur between the two $v_{m}$-elements in the type $B$ ({\em resp.} type $D$) case, which is what we wanted. Now let $i \in \{2,\ldots, m-1\}$. If  $H_{\{v_{i-1},v_{i}\}}$ contains the factor $v_{i}v_{i}$, then Proposition~\ref{prop:heaps_fullycom} and Lemma~\ref{lemma:forbidden} (2) imply that exactly two $v_{i+1}$-elements occur between the two $v_{i}$-elements. The proof follows by induction.
\end{proof}

The conclusion also holds with $i=1$ if one assumes that there exist two $v_1$-elements with no element from $\Gamma-\{v_2,\ldots,v_m,x\}$ between them.

\begin{figure}[h]
\includegraphics[width=0.70 \textwidth]{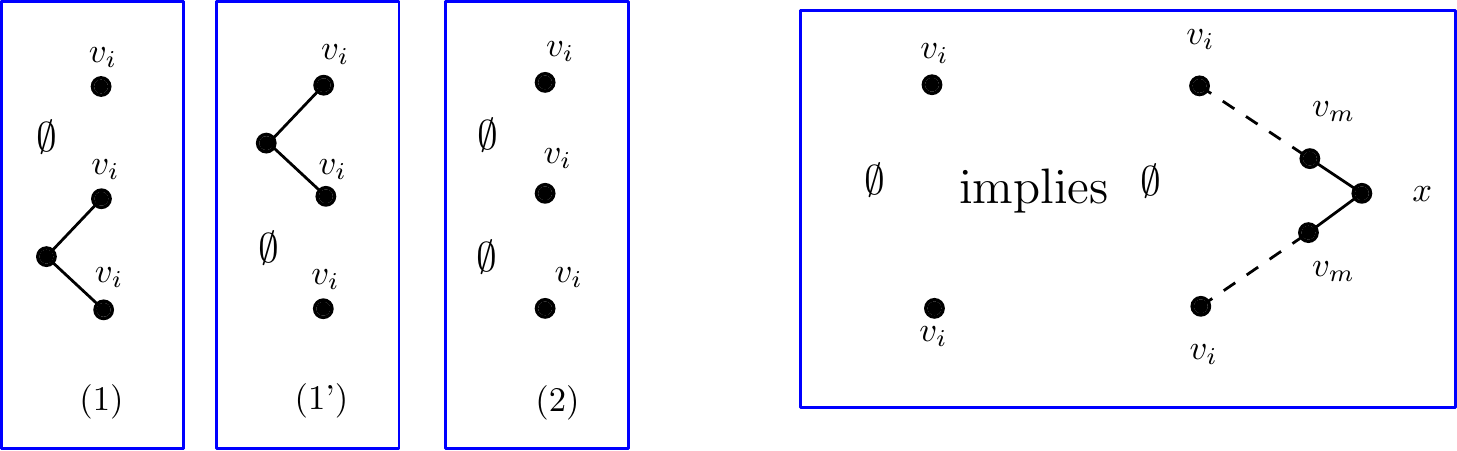}
\caption{\label{fig:forbidden} Forbidden factors of Lemma \ref{lemma:forbidden}, and an illustration of Lemma~\ref{lemma:peak} for a branch of type $B$.}
\end{figure}

Finally, the following result implies Lemma~\ref{lemma:zigzagC}, and is therefore a generalization of Ernst's~\cite[Lemma 3.3.6]{ErnstDiagramII}. Recall the definition of $P_{\rightarrow}(s_i)$ and $P_{\leftarrow}(s_i)$ in Section~\ref{sub:Caffine}, and their natural extension to types $\aff{B}$ and $\aff{D}$ in the proof of Theorem~\ref{theo:affineBDfamilles}.

\begin{lemma}[Zigzag Lemma]\label{lemma:zigzag}
Let $\Gamma$ be the Coxeter graph $\aff{C}_n$, $\aff{B}_{n+1}$ or $\aff{D}_{n+2}$. Let $H$ be a FC heap of type $\Gamma$ with an interval of the form  $P_{\rightarrow}(s_i)$ or $P_{\leftarrow}(s_i)$ for a certain $i\in \{1,\dots,n-1\}$. If $H$ contains at least another $s_i$-element, then $H$ belongs to the family (ZZ)$_{\Gamma}$.
\end{lemma}

\begin{proof}  
For symmetry reasons, we can assume that we have a peak $P_{\rightarrow}(s_i)$ whose $s_i$-elements are $s_i^{(j)}$ and $s_i^{(j+1)}$, and a third $s_i$-element $s_i^{(j+2)}$. There can be no $s_{i+1}$- element between $s_i^{(j+1)}$ and $s_i^{(j+2)}$, since this would create a factor $s_is_is_{i+1}s_i$ in the chain $H_{\{s_i,s_{i+1}\}}$, which is of the forbidden form (1') in Lemma~\ref{lemma:forbidden}. By Lemma~\ref{lemma:peak}, $s_i^{(j+1)}$ and $s_i^{(j+2)}$  therefore determine an interval of the form $P_{\leftarrow}(s_i)$.

Consider now the chain thus formed between $s_i^{(j)}$ and $s_i^{(j+2)}$: it is of the form $\H(\mathbf{c})$ where $\mathbf{c}$ is the word 
\[\mathbf{c}=s_i s_{i+1}\dots s_{n-1}u s_{n-1}\dots s_{i+1}s_i s_{i-1}\dots s_{1}t s_{1}\dots s_{i-1}s_i.\]
The chain is easily seen to be convex in $H$, and therefore there exist words $\mathbf{v},\mathbf{w}$  so that $H$ is of the form $\H(\mathbf{v}\mathbf{c}\mathbf{w})$. Now let $x$ be the first letter of $\mathbf{w}$ (if $\mathbf{w}$ is nonempty). The case $x=s_i$ is impossible since it would create a factor $s_i^2$; $x=s_j$ for $j\in\{1,\ldots,i-1\}$ is also impossible since, up to commutation, it would create a long braid $s_js_{j+1}s_j$, while $x=t$ (\emph{resp.} $x=t_1$, \emph{resp.} $x=t_2$) would create a long braid $s_1ts_1t$ (\emph{resp.} $t_1s_1t_1$, \emph{resp.} $t_2s_1t_2$). Similarly, $x=s_j$ for $j=i+2,\ldots,n-1$ and, for $i<n-1$, $x=u$ (\emph{resp.} $x=u_1$, \emph{resp.} $x=u_2$) are excluded since they would create similar long braid words after some commutations. The only remaining possibilities are $x=s_{i+1}$ if $i<n-1$, and $x=u$ for $i=n-1$ in type $\aff{C}_n$ or $\aff{B}_{n+1}$, and $x=u_1$ or $u_2$ in type $\aff{D}_{n+2}$.
  
  We need to study what happens for heaps of the form $H=\H(\mathbf{v}\mathbf{c'}\mathbf{w})$ where in type $\aff{C}_n$ \[\mathbf{c'}=u s_{n-1}\dots \dots s_{1}t s_{1}\dots s_{n-1}u.\]
In type $\aff{B}_{n+1}$, replace $t$ by $t_1t_2$ above in $\mathbf{c'}$. In type $\aff{D}_{n+2}$, there are several subcases for $\mathbf{c'}$, consisting in replacing both $u$'s independently by $u_1$, $u_2$ or $u_1u_2$. Let again $x$ be the first letter of $\mathbf{w}$ if $v$ is nonempty. Then by inspection there is only one possibility for $x$, namely $x=s_{n-1}$ in types $\aff{C}_n$ and $\aff{B}_{n+1}$; and $x=s_{n-1}$ (\emph{resp.} $x=u_1$, \emph{resp.} $x=u_2$) in type  $\aff{D}_{n+2}$ when $\mathbf{c'}$ ends with $u_1u_2$ (\emph{resp.} $u_1$, \emph{resp.} $u_2$). 

We have naturally similar results for the last letter of the prefix $\mathbf{v}$. By induction on the sum of the lengths of $\mathbf{v}$ and $\mathbf{w}$, we have then that $H$ is an element of (ZZ)$_{\Gamma}$.
\end{proof}

\section{Enumeration}
\label{sec:enum}

We now use the previous description of FC heaps in classical affine types  $\aff{C}_n$, $\aff{B}_{n+1}$ and  $\aff{D}_{n+2}$ to obtain information about their growth sequences. 

\subsection{Periodicity}

\begin{theorem}\label{theo:periodicityBCD}
Let $n\geq 2$. The growth sequence of FC elements in type $\aff{C}_n$ (\emph{resp.} $\aff{B}_{n+1}$, \emph{resp.} $\aff{D}_{n+2}$) is ultimately periodic, with period $n+1$ (\emph{resp.} $(n+1)(2n+1)$, \emph{resp.} $(n+1)$). The periodic part starts at length $\ell_0=n(n+1)/2 +3$ (\emph{resp.} $\ell_0=(n+1)(n+2)/2 +3$, \emph{resp.} $\ell_0=(n+1)(n+2)/2 +3$), except in type $\aff{C}_2$ where it starts at length $4$.
\end{theorem}

\begin{proof}
Among the five families of Definition~\ref{def:famillesCtilde}, only (ALT) and (ZZ) are infinite, so we only need to look at them in order to prove ultimate periodicity; the same is then also true for (ALT)$_{\Gamma}$ and (ZZ)$_{\Gamma}$ if ${\Gamma}=\aff{B}_{n+1}$ or $\aff{D}_{n+2}$. 

In type $\aff{C}_n$, the growth sequence of (ALT) is ultimately periodic of period $n+1$: as in Theorem~\ref{theo:periodicityA}, this is most easily seen by shifting up the paths $\varphi(H)$ from Section~\ref{sub:alternating} when the length is large enough. For (ZZ), it is clear that the growth sequence is ultimately constant equal to $2n$ for $\ell$ larger than $2n+1$, which does not modify the period and yields the ultimate period $n+1$.

In types ${\Gamma}=\aff{B}_{n+1}$ or $\aff{D}_{n+2}$, (ALT)$_{\Gamma}$ is also ultimately periodic of period $n+1$: indeed, for $\ell$ large enough, any element $H$ of (ALT) have more than one $t$-element and more than one $u$-element. Therefore by the rules of Section~\ref{sub:FCaffBD}, the corresponding $\Delta_t(H)$ and $\Delta_{t,u}(H)$ in (ALT)$_{\Gamma}$ have a constant number of elements, namely $2$ and $4$ respectively. The periodicity for (ZZ)$_{\Gamma}$ is $2n+1$ in type $\aff{B}_{n+1}$ and $n+1$ in type $\aff{D}_{n+2}$, because of the special cases where the underlying $\aff{C}_n$-heap starts with a $t$ or a $u$. This does not modify the global period in type $\aff{D}_{n+2}$, while one needs to multiply both periods in type $\aff{B}_{n+1}$. Note that $n+1$ and $2n+1$ are coprime so we cannot be more precise without further knowledge.\medskip

We now indicate the precise length where periodicity starts, for each of the three types. To achieve this, we look at both infinite families (ALT) and (ZZ) and indicate the respective lengths $\ell_1+1$ and $\ell_2+1$ at which they start being periodic, as well as the largest length $\ell_3$ for which there exists a heap from one of the three finite families. If one of the three quantities $\ell_i$ in $\{\ell_1,\ell_2,\ell_3\}$ is larger than the two others, we can deduce immediately that $\ell_0:=\ell_i+1$ is the start of periodicity.

For the family (ALT) in type $\aff{C}_n$, we can reason as in Theorem~\ref{theo:periodicityA}, and this shows  that $\ell_1$ is the largest size of a heap $H$ whose corresponding path $\varphi(H)$ has a horizontal step at height zero. One has clearly $\ell_1=n(n-1)/2$, and such a heap $H$ is shown in Figure~\ref{fig:preperiodeCtilde}, left. In types ${\Gamma}$ equal to $\aff{B}_{n+1}$ and $\aff{D}_{n+2}$, the periodicity of (ALT)$_{\Gamma}$ starts when the length is such that all heaps have an underlying $\aff{C}_n$-heap in (ALT) with at least two $t$-elements, because of the special rule for the substitutions $\Delta_t$ when there is only one $t$-element. So $\ell_1$ is the largest size of a heap whose underlying $\aff{C}_n$-heap has exactly one $t$-element, such a heap $H'$ being illustrated for type $\aff{B}_{n+1}$ in Figure~\ref{fig:preperiodeCtilde}, middle. This gives $\ell_1=(n+1)(n+2)/2+2$ in these cases.

The periodic part of (ZZ)$_{\Gamma}$ is easily determined: it starts at length $\ell_2+1$ with $\ell_2$ equal to $2n+1$ in type $\aff{C}_{n}$, $2n+3$ in type $\aff{B}_{n+1}$ and $2n+4$ in type $\aff{D}_{n+2}$.

Finally, the largest heaps in the remaining finite families  have size $\ell_3$ equal to $n(n+1)/2+2$ in type $\aff{C}_n$ and $n(n+1)/2+3 $ in types $\aff{B}_{n+1}$ and $\aff{D}_{n+2}$ (see Figure~\ref{fig:preperiodeCtilde}, right, for an example $H''$ in type $\aff{C}_n$, while in type $\aff{B}_{n+1}$ (\emph{resp.} $\aff{D}_{n+2}$) pick the unique (\emph{resp.} an)  element in $\Delta_t(H'')$ ( \emph{resp.} $\Delta_{t,u}(H'')$)).

To conclude,  one has clearly $\ell_3>\ell_1,\ell_2$ in type $\aff{C}_{n}$ for $n\geq 4$, which proves the theorem for this case (small cases $n=2$ and $3$ are checked separately). In types $\aff{B}_{n+1}$ and $\aff{D}_{n+2}$, one has now $\ell_1>\ell_2,\ell_3$ for $n\geq 2$, except for $\aff{D}_{4}$ which is easily checked separately, and this achieves the proof.

\end{proof}

\begin{figure}[!ht]
\begin{center}
\includegraphics[width=0.8\textwidth]{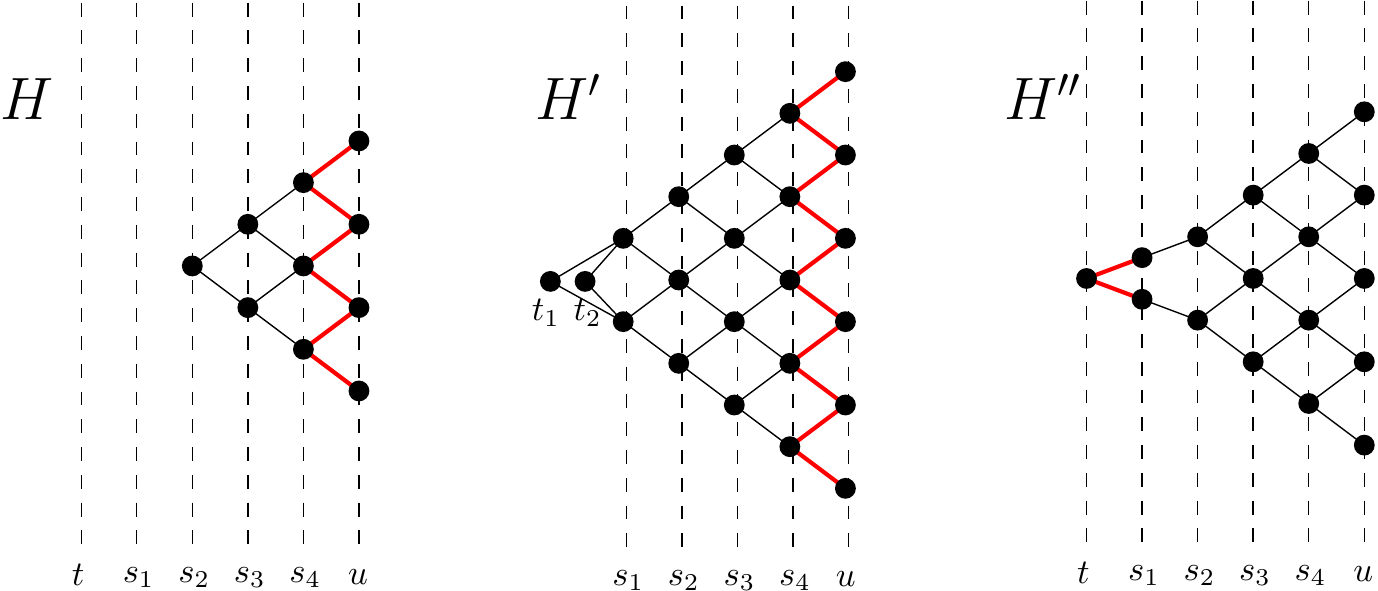}
\caption{\label{fig:preperiodeCtilde}}
\end{center}
\end{figure}

\subsection{Generating functions}
\label{sub:affines}
In this section we give functional equations which allow to compute the generating functions $W^{FC}(q)$ for types $\aff{C}_n$, $\aff{B}_{n+1}$ and  $\aff{D}_{n+2}$. These equations can be easily translated into systems of recurrence relations for the generating functions. 

Let us stress that our goal was not to derive the simplest possible expressions here, though this is certainly a worthwile investigation. What we show is that our descriptions from Section~\ref{sec:bcd} apply easily here, and in particular give us equations that are straightforwardly  programmed on a computer.

\begin{proposition}\label{prop:gf_affC}
We have the generating function:
\begin{equation}
\label{eq:Ctilde}
 \aff{C}^{FC}_n(q)= \frac{q^{n+1}\touch{\Gen}_n(q)}{1-q^{n+1}}+\touch{\Gen}^*_n(q)+\frac{2nq^{2n+2}}{1-q}+(2n-2)q^{2n+1}+2LP_n(q)+LRP_n(q),
\end{equation}
 in which the four non explicit  polynomials on the right-hand side can be evaluated as coefficients of $x^n$ in the following formulas:
\begin{align}\label{eq:touchgn}
\touch{\Gen}(x)&=\Motzbic(x)\left(1+qx\Pos(qx)\right)^2;\\
\label{eq:touchgn_s}\touch{\Gen}^*(x)&=\Motz(x)\left(1+qx\Pos(qx)\right)^2;\\
\label{eq:LP} LP(x)&=\frac{xq^2}{1-xq^2}\left(xq\Motzbic(qx)\WPos(x) + q(\Pos(qx)-1)\right);\\
\label{eq:LRP} LRP(x)&=\frac{x^2q^4}{(1-xq^2)^2}\left(x^2q^2\Motzbic^2(qx)\Motz(x) + q(\Motzbic(qx)-1)\right),
\end{align}
while  $\Pos(x)$ and $\WPos(x)$ can be explicitly computed by using 
\begin{equation}\label{systemB}
\Pos(x)=\Motzbic(x)(1+xq\Pos(qx))\quad\text{and}\quad\WPos(x)=\Motz(x)(1+xq\Pos(qx)).
\end{equation}
We finally recall that $\Motzbic(x)$ and $\Motz(x)$ are evaluated through~\eqref{relationA1} and~\eqref{eqfonctMetoile}.
\end{proposition}

\begin{proof}
Applying Theorem~\ref{theo:affineCfamilles}, it is enough to show that the sum of the generating functions for the five families in Definition~\ref{def:famillesCtilde} can be written as in \eqref{eq:Ctilde}. 

First, elements in (ALT) correspond to walks in $\Genset_n^*$ by the bijection $\varphi$ from  Theorem~\ref{theorem:walk_encoding}. Therefore their generating function is given by the first two terms in \eqref{eq:Ctilde} thanks to the decomposition

\begin{figure}[h]
\includegraphics[width=0.6 \textwidth]{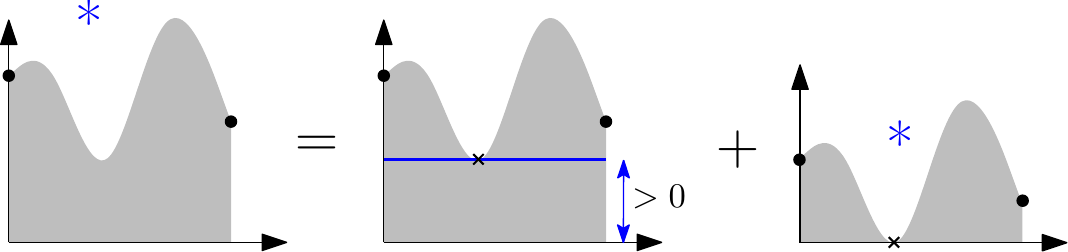}
\end{figure}

\noindent while the following one for walks in $\touch{\Genset}$ proves Equations~\eqref{eq:touchgn} and~\eqref{eq:touchgn_s}.
\begin{center}
\includegraphics[width=0.4 \textwidth]{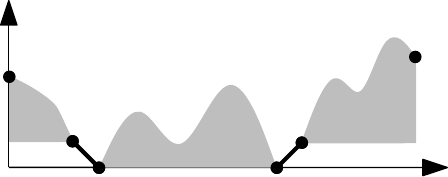}
\end{center}

Next, the zigzag heaps in (ZZ) are totally ordered, their minimal length is $2n+1$ and there are $2n-2$ of these, while for any length $>2n+1$  there are $2n$ of them. Therefore this gives the third and fourth terms in~\eqref{eq:Ctilde}.

By left-right symmetry, the generating functions for elements in (LP) and (RP) are the same. Now by definition of (LP) elements, these can be uniquely split in a peak $P_{\leftarrow}(s_j)$ and a certain heap which is alternating when one of its two $s_j$-elements is deleted. Therefore, by the bijection $\varphi$, their generating function $LP_n(q)$ is the coefficient of $x^n$ in the series
\[
\frac{xq^2}{1-xq^2}\Gen^{(1)}(x),
\]
where $\Gen^{(1)}(x)$ is the generating function for nonempty walks in $\Genset^*$ starting at height $1$. Such paths can then either touch the $x$-axis or always stay at height $\geq 1$, and the expression~\eqref{eq:LP} comes from writing the generating functions for each case.  

A similar reasoning gives us the generating function for elements in (LRP), for which we 
need to consider nonempty walks in $\Genset^*$ starting and ending at height $1$.\\
Finally, equations \eqref{systemB} are easily computed by decomposing paths in $\Posset_n$ and $\Posset_n^*$ according to their first return on the $x$-axis.
\end{proof}

In the same way, we prove the following expression for $\aff{B}^{FC}_{n+1}(q)$.

\begin{proposition}\label{prop:afftypeB}
We have the generating function:
\begin{equation}
\label{eq:Btilde}
\aff{B}^{FC}_{n+1}(q)= \frac{2q^{n+1}\touch{\Gen}_n(q)}{1-q^{n+1}}+T_n(q)+ZZ_{\aff{B}_{n+1}}(q)+q(LP_n(q)+LRP_n(q))+RP^{\Delta_t}_n(q),
\end{equation}
in which we have
\begin{equation}\label{eq:ZZ_B}
 ZZ_{\aff{B}_{n+1}}(q)=\frac{(2n+3)q^{2n+4}}{1-q}+\frac{q^{2(2n+1)}}{1-q^{2n+1}}+(2n+2)q^{2n+3}+(2n-2)q^{2n+2},
\end{equation}
while the polynomials $T_n(q)$ and $ RP^{\Delta_t}_n(q)$ are the respective coefficients of $x^n$ in the following series:
\begin{align}
\label{eq:ALT_B}
T(x)&=\WPos(x)+2(\touch{\Gen}^*(x)-\WPos(x))+xq^2\Motzbic(qx)\WPos(x)+q^2\Pos(qx);\\
 RP^{\Delta_t}(x)&=\frac{xq^2}{1-xq^2}\left[xq\Motzbic(qx)\Motz(x) + 2qx\Motzbic(qx)(\Pos^*(x)-\Motz(x))\right.\nonumber\\
&\left.+2q(\Pos(qx)-1)+q^2(\Motzbic(qx)-1)+q^3x^2\Motzbic^2(qx)\Motz(x)\right]\nonumber\\
\label{eq:RP_B} &+\frac{q^3+4q^2+2q}{1-xq^2}.
\end{align}
\end{proposition}

\begin{proof}
Elements in (ALT)$_{\aff{B}}$ are by definition of the form $\Delta_t(H)$ where $H$ is a heap of type $\aff{C}_{n}$ in (ALT). $\Delta_t$ acts differently on heaps with $0$, $1$ or at least two $t$-elements, and distinguishing these cases yields the two first terms of~\eqref{eq:Btilde} (and the expression~\eqref{eq:ALT_B}) as the generating function  for elements in (ALT)$_{\aff{B}}$: recall indeed that $\Gen_n^*(q)$ can be written as $\touch{\Gen}_n^*(q)+q^{n+1}\touch{\Gen}_n(q)/(1-q^{n+1})$. Moreover, it can be computed thanks to Equations~\eqref{eq:touchgn} and \eqref{eq:touchgn_s}. 

Heaps in (ZZ)$_{\aff{B}_{n+1}}$ have minimal length  $2n+2$. Moreover, there are $2n-2$ (\emph{resp.} $2n+2$) such elements of length $2n+2$ (\emph{resp.}  $2n+3$). Then for any length $\ell>2n+3$  there are $2n+3$ elements in (ZZ)$_{\aff{B}_{n+1}}$ unless $\ell$ is divisible by $2n+1$, in which case there are $2n+4$ of them. Therefore the generating function for elements in (ZZ)$_{\aff{B}}$ is given by~\eqref{eq:ZZ_B}.

Elements in (LP)$_{\aff{B}}$ or (LRP)$_{\aff{B}}$ are obtained from elements in (LP) or (LRP) by replacing $t$ by $t_1t_2$, so the corresponding generating functions are simply multiplied by $q$. Elements in (RP)$_{\aff{B}}$ are also deduced from those of type $\aff{C}_{n}$ in (RP), but one has to distinguish whether there are $0$, exactly $1$ or more than $1$ occurrences of $t$. It is then routine to write the Equation~\eqref{eq:RP_B}, whose last term is a correction accounting for the special heap occurring in the definition of $\Delta_t$.
\end{proof} 

Finally, our method also yields the following expression for $\aff{D}^{FC}_{n+2}(q)$.

\begin{proposition}\label{prop:afftypeD}
We have the generating function:
\begin{equation}
\label{eq:Dtilde}
\aff{D}^{FC}_{n+2}(q)=  \frac{4q^{n+1}\touch{\Gen}_n(q)}{1-q^{n+1}}+U_n(q)+ZZ_{\aff{D}_{n+2}}(q)+q^2LRP_n(q)+2qRP^{\Delta_t}_n(q),
\end{equation}
in which we have
\begin{equation}\label{eq:ZZ_D}
 ZZ_{\aff{D}_{n+2}}(q)=\frac{(2n+6)q^{2n+5}}{1-q}+\frac{2q^{3(n+1)}}{1-q^{n+1}}+(2n+4)q^{2n+4}+(2n-2)q^{2n+3},
\end{equation}
while the polynomial $U_n(q)$ is the coefficient of $x^n$ in the following series:
\begin{align}
U(x)&=\Motz(x)+4x^2q^2\Motz(x)\Pos(qx)^2\nonumber \\
&+q^2[q\Motzbic(qx)+\Motz(x)(xq\Motzbic(qx))^2]+4xq\Motz(x)\Pos(qx)\nonumber \\
\label{eq:ALT_D} &+2xq^2\Motz(x)\Motzbic(qx)+4q^2\Pos(qx)+4x^2q^3\Motz(x)\Motzbic(qx)\Pos(qx).
\end{align}

\end{proposition}

\begin{proof}
Finite families are easily dealt with thanks to the proofs of types $\aff{C}_n$ and $\aff{B}_{n+1}$. The analysis of the family (ZZ)$_{\aff{D}_{n+2}}$ is similar to (ZZ)$_{\aff{B}_{n+1}}$. For the family (ALT)$_{\aff{D}_{n+2}}$, one needs to consider if the underlying FC heap $H$ has $0,1$ or at least two $t$-elements, and the same with $u$-elements. Thus one needs to compute the generating functions for the corresponding families of paths $\varphi(H)$, which is elementary but tedious, and we omit the details here. 
\end{proof}

\subsection{Mean values}
\label{sub:MeanValuesBCD}
From these complicated generating functions it is not easy to obtain simple formulas for individual coefficients. In the article~\cite{JouhetNadeau}, more will be said about these, at least for the purely periodic part of the sequences.

Nevertheless one can compute already precisely what the {\em mean value} of these coefficients is. As we shall see in Section~\ref{sub:TL}, this tells us in particular how fast the corresponding Temperley--Lieb algebras grow with respect to their usual generators.

\begin{proposition}
\label{prop:meanBCD}
For $W$ of type $\aff{C}_n$ (\emph{resp.} $\aff{B}_{n+1}$, \emph{resp.} $\aff{D}_{n+2}$), the mean value $\mu_W$ of the growth sequence of FC elements is given by $2n+\frac{4^n}{n+1}$ (\emph{resp.} $ 2n+3+\frac{1}{2n+1}+\frac{2\cdot 4^n}{n+1}$, \emph{resp.} $ 2n+6+\frac{2+4^{n+1}}{n+1} $).
\end{proposition}

\begin{proof}
The proof goes as in Proposition~\ref{prop:meanA} and uses the simple Lemma~\ref{lemma:periodic}. By inspecting the generating functions involved, one checks that one only has to prove $\touch{\Gen}_n(1)=4^n$. Given any path in $\touch{\Genset}$, shift it first so that it starts at the origin. This is a bijection from $\touch{\Genset}$ to paths of length $n$ which start at the origin and are not necessarily positive, of which there are clearly $4^n$ since steps can then be chosen independently.
\end{proof}

\subsection[Type B]{Types $B_n$ and $D_{n+1}$}
\label{sub:finBD}
FC heaps of type $B_n$ (\emph{resp.} $D_{n+1}$) embed easily as FC heaps of type $\aff{C}_n$ (\emph{resp.} $\aff{B}_{n+1}$) which have no $u$-element. For type $B_n$, they correspond to elements of type $\aff{C}_n$ which belong to one of the two families (ALT) or (LP) and have at most one $s_{n-1}$-element; the same holds for type $D_{n+1}$ after applying the substitution map $\Delta_t$.

\begin{proposition}
\label{typeB} The generating polynomials $B_n^{FC}(q)$ and $D_{n+1}^{FC}(q)$ for $n\geq 2$ are the coefficients of $x^n$ in the following series:
 \begin{align}
 B^{FC}(x)&=\WPos(x) + \frac{x^2q^3}{1-xq^2} \Motz(x)\Motzbic(qx);\\
 D^{FC}(x) &= 2 \WPos(x) -\Motz(x) + \frac{xq^2}{1-xq^2}\Motz(x)\Motzbic(qx).
\end{align}
\end{proposition}

Note that the FC elements of type ${B}_n$ corresponding to alternating heaps are called \emph{fully commutative top elements of ${B}_n$} by Stembridge in~\cite{St2}, and \emph{commutative elements of the Weyl group ${\mathcal C}_n$} by Fan in~\cite{Fan} (see also \cite[Remark~2.4]{St3}). Therefore $\WPos(x)$  gives a generating function for these particular elements.


\section[Exceptional types]{Exceptional types}
\label{sec:excep}

In this section we study FC elements in the exceptional types. We start with the finite case this time, in which everything can be left to the computer. In the affine case, there are two types which have a finite number of FC elements, and they can be dealt with as in the finite case. The analysis is more subtle in the remaining three affine types.

\subsection{Finite case}
\label{sub:finExcept}

The exceptional types are $I_2(m),H_3,H_4,F_4,E_6,E_7,$ and $E_8$, whose Coxeter graphs were shown in Figure~\ref{fig:finite_diagrams}. For the dihedral group $I_2(m)$, only the element of maximal length is not FC. For the remaining types, we used the GAP package GBNP to find the generating polynomials $W^{FC}(q)$ in each case, as mentioned in the introduction and explained in Section~\ref{sub:TL}.

\begin{eqnarray*}
I_2(m)^{FC}(q)&=&1+2q+2q^2+\cdots+2q^{m-1}=1+{2q(1-q^{m-1})}/(1-q);\\
H_3^{FC}(q)&=&q^{10}+2q^{9}+3q^{8}+4q^{7}+5q^{6}+7q^{5}+7q
^{4}+6q^{3}+5q^{2}+3q+1;\\
H_4^{FC}(q)&=&q^{16}+2q^{15}+3q^{14}+4q^{13}+8q^{12}+12q^{11}+16q^{10}+18q^{9}+20q^{8}+21q^{7}\\
&&+23q^{6}+21q^{5}+18q^{4}+14q^{3}+9q^{2}+4q+1;\\
F_4^{FC}(q)&=&2q^{10}+4q^{9}+8q^{8}+12q^{7}+16q^{6}+18q^{5}+
18q^{4}+14q^{3}+9q^{2}+4q+1;\\
E_6^{FC}(q)&=&2q^{16}+4q^{15}+6q^{14}+8q^{13}+14q^{12}+27q^{
11}+40q^{10}+53q^{9}+76q^{8}\\
&&+91q^{7}+99q^{6}+95q^{5}+75q^{4}+45q^{3}+20q^{2}+6q+1;\\
E_7^{FC}(q)&=&q^{27}+2q^{26}+3q^{25}+4q^{24}+5q^{23}+8q^{22}+
11q^{21}+14q^{20}+17q^{19}\\
&&+22q^{18}+39q^{17}+57q^{16}+73q^{15}+91q^{14}+125q^{13}+159q^{12}+198q^{11}\\
&&+236q^{10}+275q^{9}+297q^{8}+298q^{7}+273q^{6}+216q^{5}+140q^{4}+71q^{3}\\
&&+27q^{2}+7q+1;\\
E_8^{FC}(q)&=&15q^{29}+30q^{28}+43q^{27}+56q^{26}+69q^{25}+83q^{24}+113q^{23}
+143q^{22}\\
&&+171q^{21}+205q^{20}+259q^{19}+319q^{18}+387q^{17}+457q^{16}
+527q^{15}\\
&&+609q^{14}+701q^{13}+794q^{12}+867q^{11}+924q^{10}+936{
q}^{9}+897q^{8}\\
&&+796q^{7}+631q^{6}+427q^{5}+238q^{4}+105q^{3}+35q^{2}+8q+1.
\end{eqnarray*}

Note that the number of FC elements in a Coxeter group may be finite even though the group itself is infinite: Stembridge~\cite{St1} discovered that there are  three families $E_n (n>8), F_n (n>4), H_n (n>4)$ of infinite groups with a finite number of FC elements. Extending Stembridge~\cite{St3} with similar walk techniques, it is possible to enumerate such elements according to their length.

\subsection{Affine case}
\label{sub:affExcept}

The exceptional irreducible affine types are $\aff{E}_6,\aff{E}_7,\aff{E}_8,\aff{F}_4,$ and $\aff{G}_2$, and shown in Figure~\ref{fig:affinediagrams}. The number of FC elements in types $\aff{F}_4$ or $\aff{E}_8$ is finite, since they correspond in Stembridge's classification~\cite[Theorem 4.1]{St1} to types ${F}_5$ and ${E}_9$.

\begin{align*}
\aff{F}_4^{FC}(q)&=q^{18}+2q^{17}+3q^{16}+6q^{15}+9q^{14}+13q^{13}+18q^{12}+27q^{11}+35q^{10}\\
&+44q^{9}+52q^{8}+57q^{7}+57q^{6}+52q^{5}+41q^{4}+27q^{3}+14q^{2}+5q+1;\\
\aff{E}_8^{FC}(q)&=q^{44}+2q^{43}+3q^{42}+4q^{41}+5q^{40}+6q^{39}+9q^{38}\\
&+16q^{37}+23q^{36}+ 32q^{35}+47q^{34}+68q^{33}+97q^{32}+138q^{31}+256q^{30}\\
&+368q^{29}+462q^{28}+562q^{27}+669q^{26}+786q^{25}+916q^{24}+1065q^{23}+1199q^{22}\\
&+1355q^{21}+1529q^{20}+1728q^{19}+1916q^{18}+2118q^{17}+2298q^{16}+2494q^{15}\\
&+2693q^{14}+2866q^{13}+2970q^{12}+3002q^{11}+2923q^{10}+2710q^{9}\\
&+2354q^{8}+1862q^{7}+1297q^{6}+770q^{5}+378q^{4}+148q^{3}+44q^{2}+9q+1.
\end{align*}

\begin{figure}[!ht]
\begin{center}
\includegraphics[height=8cm]{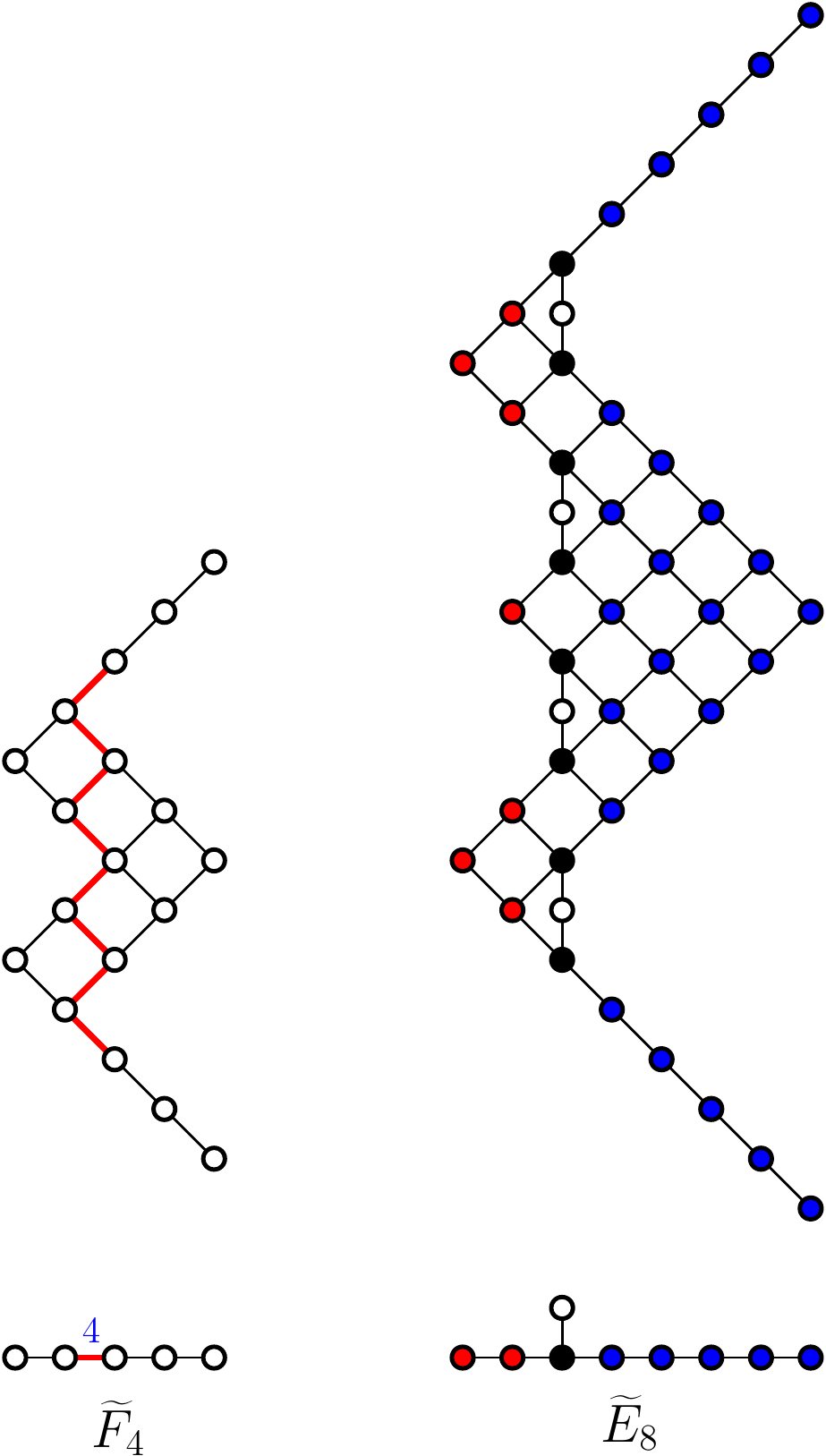}
\caption{The unique longest FC elements in types $\aff{F}_4$ and $\aff{E}_8$.
\label{fig:exceptional_affine_finite}
}
\end{center}
\end{figure}

Notice that the leading coefficient of each of these polynomials is equal to 1; the corresponding longest heaps are depicted in Figure~\ref{fig:exceptional_affine_finite}.
\smallskip

In the remaining three types, we will make repeated use of the following lemma, whose proof presents no difficulty. A {\em simple branch} in a Coxeter graph $\Gamma$ is a path $t,s_1,s_2,\ldots,s_m$, where all edges have label $3$ (see Figure~\ref{f:simplebranch}). The next lemma corresponds essentially to Lemmas 5.3 and 5.4 in \cite{St1}. Its proof is an easy consequence of Lemma~\ref{lemma:local}.

\begin{lemma}[Stembridge \cite{St1}]
\label{lemma:stemb_simple}
Let $t,s_1,s_2,\ldots,s_m$ be a simple branch in a  diagram $\Gamma$, where $s_m$ has degree $1$, and $H$ a FC heap of type $\Gamma$. Then the following hold true:
\begin{enumerate}
\item there is at least one $t$-element in $H$ between any two $s_1$-elements.
\item If $H_{\{s_1,t\}}$ contains an alternating factor with $l$ $s_1$-elements, then $l\leq m$.
\end{enumerate}
 \end{lemma}

 \begin{figure}
 \begin{center}
\includegraphics[width=0.5\textwidth]{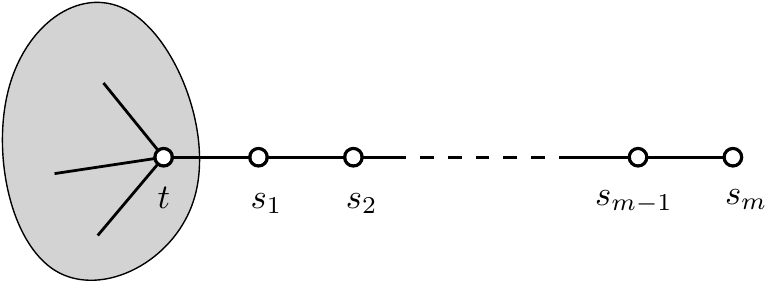}\label{f:simplebranch}
\end{center} 
\caption{A simple branch.}
\end{figure}
\smallskip

To compute the growth sequence in each case, we need to prove that, starting from a certain length $\ell_0$+1, it has period $P$. Then the knowledge of all terms up to length $\ell_0+P$ suffices. We explain in Section~\ref{sub:TL} how to use the \texttt{GAP} package \texttt{GBNP} to perform this.

\subsubsection{Type $\aff{G}_2$}

The key lemma is the following.

\begin{lemma}
Let $H$ be a FC heap of type $\widetilde{G}_2$ with more than four $t$-elements. Then the interval between extremal $t$-elements is a factor of the word $(tsutu)^N$ for a certain integer $N>0$.
\end{lemma}

An example of such a factor is illustrated in Figure~\ref{fig:exceptional_affine}, left.

\begin{proof}
 Let $t^{(1)}<\dots<t^{(k)}$ be the $t$-elements in $H$. They determine $k-1$  intervals $H_1,\ldots, H_{k-1}$ with $H_i:=]t^{(i)},t^{(i+1)}[$. Define also two other intervals $H_0$ (\emph{resp.} $H_k$) corresponding  to elements smaller than $t_1$ (\emph{resp.} larger than $t^{(k)}$). By Definition~\ref{defi:heaps} it is clear that these $k+1$ intervals are pairwise disjoint and contain all $s$- and $u$-elements.

 Consider first the chain $H_{\{t,u\}}$. No two $u$-elements in $H_{\{t,u\}}$ can be consecutive by Proposition~\ref{prop:heaps_fullycom}(b). No two $t$-elements can be consecutive either: indeed, if $t^{(i)}$ and $t^{(i+1)}$ were not separated by a $u$-element, then by Proposition~\ref{prop:heaps_fullycom}(b) they would be separated by exactly one $s$-element, which then contradicts Proposition~\ref{prop:heaps_fullycom}(a). We have thus proven that $H_{\{t,u\}}$ is alternating.

Now consider the chain $H_{\{s,t\}}$. Here $t,s$ constitutes a simple branch with $m=1$, so by Lemma~\ref{lemma:stemb_simple} there is at most one $s$-element in each $H_i$, and there are no $s$-elements in two consecutive $H_i$'s.

Now assume $k\geq 4$. Before going further, notice the important fact that there is but a finite number of heaps with $k\leq 3$, as follows easily from the previous reasoning; in fact the maximal size of such heaps is easily bounded above by $9$. Now suppose  that there exists $i\in\{1,\ldots,k-2\}$ such that both $H_i$ and $H_{i+1}$ contain no $s$-element. Then because $H_{\{t,u\}}$ is alternating and $k\geq 4$,  $H$ contradicts Proposition~\ref{prop:heaps_fullycom}(a) with a convex $\{t,u\}$-chain of length $6$.
So $s$-elements must occur in every other $H_i$ for $i\in\{1,\ldots,k-1\}$, which proves the proposition.
\end{proof}

It is then easy to prove that the growth sequence for FC elements in type $\aff{G}_{2}$ has periodicity $P=5$, starting from  length $\ell_0+1=8$, which corresponds to the repetition of the length $5$-pattern $tsutu$. Then the whole series is  
\begin{align*}
\aff{G}_2^{FC}(q)&=1+3q+5q^{2}+6q^{3}+7q^{4}+9q^{5}+9q^{6}+8q^{7}\\
&+7q^{8}+7q^{9}+8q^{10}+7q^{11}+7q^{12}\\
&+7q^{13}+7q^{14}+8q^{15}+7q^{16}+7q^{17}\\&+\cdots
\end{align*}


\begin{figure}[!ht]
\begin{center}
\includegraphics[height=8cm]{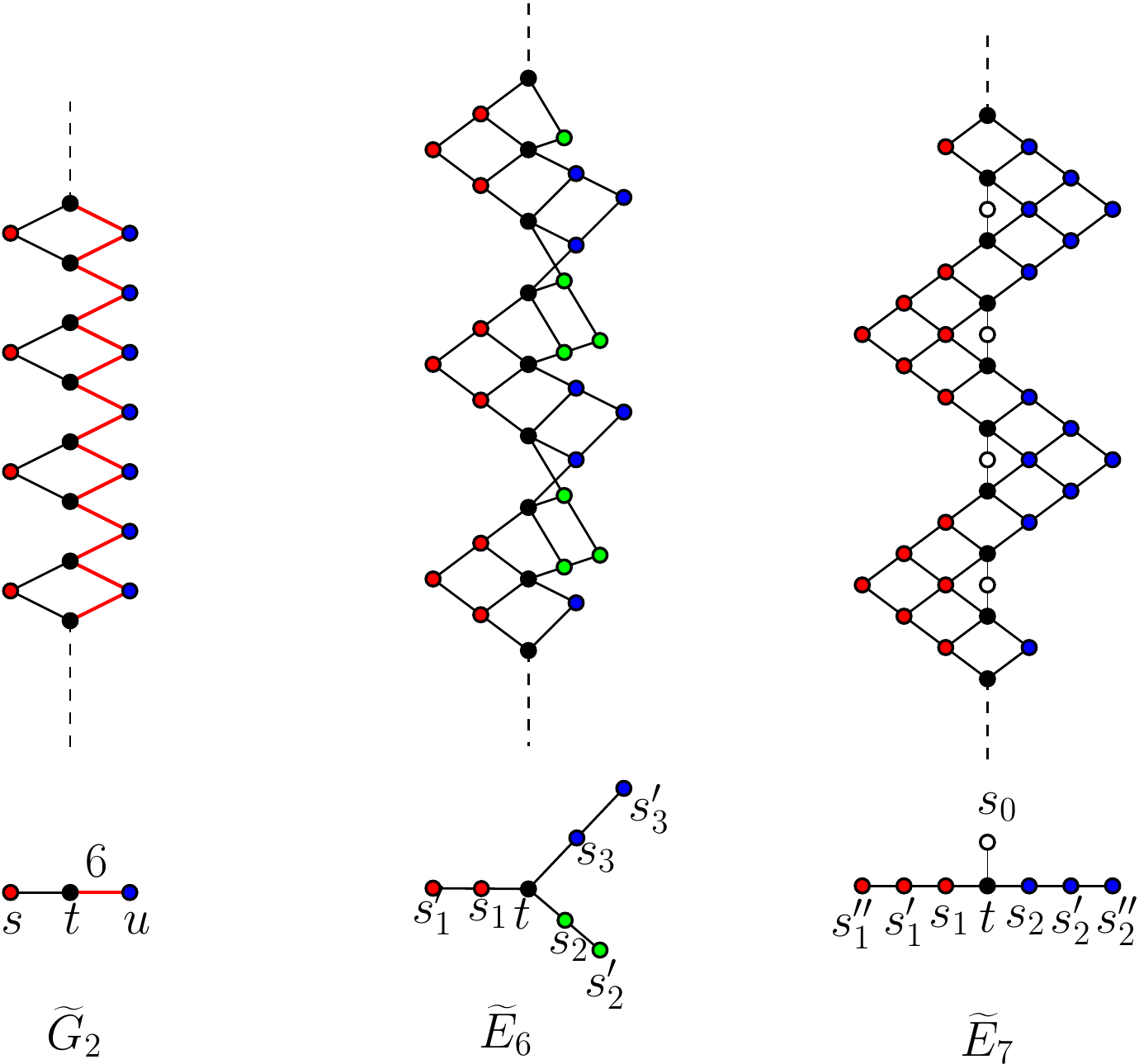}
\caption{ Central sections of long FC elements.
\label{fig:exceptional_affine}
}
\end{center}
\end{figure}

\subsubsection{Type $\aff{E}_6$}

In this case, the key lemma is the following.

\begin{lemma}
Let $H$ be a FC heap of type $\widetilde{E}_6$ with more than five $t$-elements. Then the interval between extremal $t$-elements occurs as an interval in $\H(\mathbf{w})$ with
\[\mathbf{w}=\left((s_1ts_1's_1)(s_2ts_2's_2)(s_3ts_3's_3)\right)^N\quad\text{or}\quad\mathbf{w}=\left((s_1ts_1's_1)(s_3ts_3's_3)(s_2ts_2's_2)\right)^N,
\]
for a certain integer $N>0$.
\end{lemma}

Such a factor is illustrated in Figure~\ref{fig:exceptional_affine}, center.


\begin{proof}
Let $H$ be a FC heap with $k\geq 0$ $t$-elements $t^{(1)}<\dots<t^{(k)}$. By Proposition~\ref{prop:heaps_fullycom}, there are at least two elements with labels in $\{s_1,s_2,s_3\}$ between $t^{(i)}$ and $t^{(i+1)}$ for any $i<k$. But by Lemma~\ref{lemma:stemb_simple}, there is at most one $s_j$-element between $t^{(i)}$ and $t^{(i+1)}$  for $j=1,2,3$. It follows that between two successive $t$-elements, there are either (a) two elements labeled with distinct labels in $\{s_1,s_2,s_3\}$ or (b) one $s_j$-element for $j=1,2,3$.

Assume Case (b) occurs in $H$, say between $t^{(i)}$ and $t^{(i+1)}$; we will show that $k=2$ or $3$ in this case. If there exists no other $t$-element, we are done: otherwise we can assume $t^{(i+2)}$ is defined, up to considering the dual heap. There are at least two elements between $t^{(i+1)}$ and $t^{(i+2)}$ as seen above; by the symmetry of the Coxeter diagram, we can assume they have labels $s_1$ and $s_2$. No $t$-element can occur below $t^{(i)}$ or above $t^{(i+2)}$: indeed this would force either a $s_1$- or $s_2$-element to occur there, which is impossible by Lemma~\ref{lemma:stemb_simple}. Therefore $i=1$ and $k=3$ in this case.

Now we assume $k\geq 4$, so that case (a) holds: between two successive $t$-elements, there are two elements with distinct labels in $\{s_1,s_2,s_3\}$. Consider first $t^{(1)},t^{(2)}$ and assume without loss of generality that these labels are $\{s_1,s_2\}$. The possible labels in $]t^{(2)},t^{(3)}[$ are then either $\{s_1,s_3\}$ or $\{s_2,s_3\}$; indeed  $\{s_1,s_2\}$ cannot occur since the same argument as in the previous paragraph shows that this forces $k=3$. We can assume these labels are $\{s_2,s_3\}$ thanks to the Coxeter graph symmetry. Now the possible labels in $]t^{(3)},t^{(4)}[$ are $\{s_1,s_2\}$ are $\{s_1,s_3\}$, but the first possibility is excluded by Lemma~\ref{lemma:stemb_simple}. Following the analysis, we see that the labels are alternatively $\{s_1,s_2\},\{s_2,s_3\},\{s_3,s_1\},\{s_1,s_2\}$ and so on. 

So the $s_j$-elements fit with the statement of the lemma; it remains to see that the $s'_j$-elements do too. By Lemma~\ref{lemma:stemb_simple}, there is at most one $s'_j$-element between two consecutive $s_j$-elements, and exactly one if these $s_j$-elements occur in $]t^{(i)},t^{(i+1)}[$ and $]t^{(i+1)},t^{(i+2)}[$ for a certain $i$. To contradict the statement of the lemma, there must therefore be a $s'_j$-element between two $s_j$-elements occurring in $]t^{(i)},t^{(i+1)}[$ and $]t^{(i+2)},t^{(i+3)}[$ respectively. It is easy to see that in this case, no other $t$-element can occur, so that $i=1$ and $k=4$ is the only possibility. Therefore the lemma holds since it supposes $k\geq 5$.
\end{proof}

\begin{figure}
\begin{center}
\includegraphics{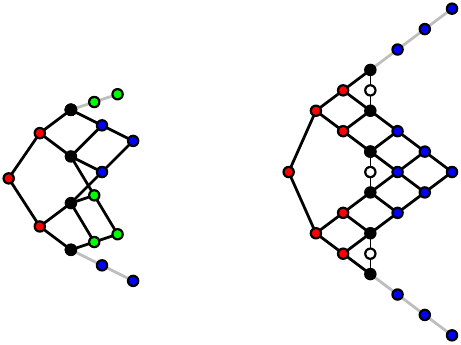}
\caption{ Special elements in type $\aff{E}_6$ and $\aff{E}_7$.
\label{fig:specialE6E7}
}
\end{center}
\end{figure}

By analysing the previous proof, one checks that the longest length among FC elements with $k\leq 4$ is $\ell_0=17$, an example being given in Figure~\ref{fig:specialE6E7}, left. Also there is only a finite number of heaps with a given extremal $t$-interval for $k\geq 5$. So periodicity of the growth sequence for FC elements of type $\aff{E}_6$ is $P=4$, and the series is
\begin{align*}
  \aff{E}_6^{FC}(q)&=1+7q+27q^{2}+71q^{3}+141q^{4}+220q^{5}+284q^{6}  +322q^{7}+338q^{8}+327q^{9}\\
  &+298q^{10}+269q^{11}+232q^{12}+177q^{13}+150q^{14}+138q^{15}+138q^{16}+126q^{17}\\
  &+120q^{18}+120q^{19}+126q^{20}+120q^{21}\\
  &+120q^{22}+120q^{23}+126q^{24}+120q^{25}+\cdots
\end{align*}

\subsubsection{Type $\aff{E}_7$}

\begin{lemma}
Let $H$ be a FC heap of type $\widetilde{E}_7$ with $k\geq 7$ $t$-elements. Then the interval between extremal $t$-elements occurs as an interval in $\H(\mathbf{w})$ with
\[\mathbf{w}=\left((s_1ts'_1s_0s'_1s''_1ts'_1s_1)(s_2ts'_2s_0s'_2s''_2ts'_2s_2)\right)^N,
\]
for a certain integer $N>0$.

\end{lemma}

We will sketch the proof here, since it closely resembles the one for $\aff{E}_6$.

\begin{proof}
Let $H$ be  a FC heap of type $\widetilde{E}_7$ with $k\geq 0$ $t$-elements $t^{(1)}<\dots<t^{(k)}$. Between two successive $t$-elements, there are either (a) two elements with distinct labels in $\{s_0,s_1,s_2\}$ or (b) three elements with labels $s_0,s_1$ and $s_2$. Case (b) occurs only if $k\leq 4$, as one checks easily. 
 Now assume $k\geq 5$ so that between two successive $t$-elements, there are two elements with distinct labels in $\{s_0,s_1,s_2\}$. The label $s_0$ cannot appear in consecutive intervals, and if it does not appear in two consecutive intervals, then $k\leq 5$ is forced. For $k\geq 6$ it is then easy to check that the $s$-labels fit with the statement of the lemma, and this holds as well for the $s'$-labels; it remains to see that the $s''$-elements do too. For $k=6$ there is a possibility for such an $s''$-element to contradict the lemma, see Figure~\ref{fig:specialE6E7}, right, but such things cannot happen for  $k\geq 7$, which proves the lemma.
\end{proof}

By analysing the previous proof, one checks that the longest length among FC elements with $k\leq 6$ which do not obey the conclusion of the lemma is $28$, an example being given in Figure~\ref{fig:specialE6E7}, right. Also there is only a finite number of heaps with a given extremal $t$-interval for $k\geq 7$. So periodicity of the growth sequence for FC elements of type $\aff{E}_6$ starts at $\ell_0+1=29$ and the function has then period $P=9$.

\begin{align*}
&\aff{E}_7^{FC}(q)=1+8q+35q^{2}+105q^{3}+238q^{4}+428q^{5}+634q^{6}+806q^{7}+918q^{8}\\
&+976q^{9}+979q^{10}+940q^{11}+873q^{12}+802q^{13}+713q^{14}+623q^{15}+546q^{16}\\
&+473q^{17}+390q^{18}+295q^{19}+256q^{20}+228q^{21}+212q^{22}+200q^{23}+188q^{24}\\
&+184q^{25}+180q^{26}+180q^{27}+176q^{28}\\
&+174q^{29}+174q^{30}+174q^{31}+174q^{32}+174q^{33}+174q^{34}+174q^{35}+176q^{36}+174q^{37}\\
&+174q^{38}+174q^{39}+174q^{40}+174q^{41}+174q^{42}+174q^{43}+174q^{44}
+176q^{45}+174q^{46}\\
&+\cdots
\end{align*}

\section{Applications and further questions}
\label{sec:further}

In this section we first give a direct application of our results to the growth of Temperley--Lieb algebras, as was announced in the introduction. Then in Sections~\ref{sub:min} to ~\ref{sub:gencox}, we announce some work in progress concerning natural extensions of the results. Finally a few questions which we believe deserve further study are listed in Section~\ref{sub:further}.

\subsection{Temperley--Lieb algebras and their growth}
\label{sub:TL}

Consider the ring $\mathcal{A}=\mathbb{Z}[\mathfrak{q},\mathfrak{q}^{-1}]$; here we use $\mathfrak{q}$ instead of $q$ to avoid confusion with the variable in our generating functions. For $W$ a Coxeter group with Coxeter matrix $M=(m_{st})_{s,t\in S}$, the associated  {\em Hecke algebra} $\mathcal{H}(W)$ is given by generators $T_s$ and relations 
\begin{align*}
T_s^2&=(\mathfrak{q}-1)T_s+\mathfrak{q}\mathbf{1}\quad&\text{for }s\in S;\\
\underbrace{T_sT_tT_s\cdots}_{m_{st}} &= \underbrace{T_tT_sT_t\cdots}_{m_{st}}\quad&\text{for }s\neq t\in S.
\end{align*}

For any $w\in W$, define $T_w\in \mathcal{H}(W)$ by picking any reduced decomposition $s_{i_1}\cdots s_{i_m}$ for $w$ and setting $T_w:=T_{s_{i_1}}\cdots T_{s_{i_m}}$, and $T_e=\mathbf{1}$. These elements $T_w$ then form a basis of $ \mathcal{H}(W)$ (see for instance~\cite{Humphreys}).

 The {\em generalized Temperley--Lieb algebra} $\TL(W)$ is defined as the quotient of $\mathcal{H}(W)$  by the ideal generated by the elements 
\[\sum_{w\in W_{s,t}}T_w,\quad\text{if}\quad m_{st}\geq 3,\]
where $W_{s,t}$ is the (dihedral) subgroup generated by $s$ and $t$. For instance if $m_{st}=3$ the element is $T_sT_tT_s+T_sT_t+T_tT_s+T_t+T_s+\mathbf{1}$. Let $b_w$ be the image of $T_w$ in $\TL(W)$. Then the elements $b_w$, for $w\in W^{FC}$, form a basis of $\TL(W)$ (see ~\cite[Theorem 6.2]{Graham}).

Consider now the natural filtration $\TL(W)_0\subset \TL(W)_1\subset \cdots$ of $\TL(W)$, where $\TL(W)_{\ell}$ is the linear span in $\TL(W)$ of all products $b_{s_{i_1}}\cdots b_{s_{i_k}}$ with $k\leq \ell$. A linear basis for $\TL(W)_\ell$ is clearly given by $(b_w)_w$, where $w$ lies in the set of  all FC elements of length at most $\ell$. Let the {\em growth} of $\TL(W)$ be $G^W:\ell\mapsto \dim \TL(W)_{\ell}$, so that $G^W(\ell)$ is the number of FC elements of length at most $\ell$ in $W$.
Now let $W$ be an irreducible affine group with infinitely many FC elements. Recalling the definition of the mean value $\mu_W$ given after Corollary B, we get the following result.

\begin{theorem}
\label{theo:growthTL}
 For any affine group with infinitely many FC elements, the algebra $\TL(W)$  has {\em linear growth}: one has the asymptotic equivalent $G^W(\ell)\sim \mu_W \ell$ when $\ell$ tends to infinity.
\end{theorem}

We refer to the books~\cite{GrowthAlgebrasBook,Ufnarovskij} for more informations on growth of algebras; in particular a less precise version of Theorem~\ref{theo:growthTL} is that such algebras have {\em Gelfand--Kirillov dimension} equal to $1$. In the simply laced case, it was also noticed in~\cite{Ukraine_simple} that the growth is linear.

 Define the {\em nil Temperley--Lieb algebra} ${\rm nTL}(W)$ as the graded algebra associated to $\TL(W)$: by definition, its $\ell$th grade component is given by $\TL(W)_\ell/\TL(W)_{\ell-1}$, and the multiplication is inherited from $\TL(W)$. It is easily seen to have the presentation with generators $u_s$ and relations 
 \[\begin{cases}
 u_s^2=0;\\
 \underbrace{u_su_tu_s\cdots}_{m_{st}}=0\quad\text{ if}& m_{st}\geq 3;\\
  u_su_t=u_tu_s\quad\text{ if}& m_{st}=2.
 \end{cases}\]
  This algebra seems to have been studied only for type $A_{n-1}$ by Fomin and Greene in~\cite{FominGreene} and for type $\aff{A}_{n-1}$ by Postnikov in~\cite{Postnikov}.  Either from its definition or the presentation, one sees that the $\ell$th graded component of ${\rm nTL}(W)$ has a basis $(u_w)$ indexed by FC element of length $\ell$, and we have the following consequence.
  
\begin{corollary}
The Hilbert series of  ${\rm nTL}(W)$ is equal to $W^{FC}(q)$. 
\end{corollary}
 Now for any $W$, ${\rm nTL}(W)$ is a finitely presented, graded algebra: for any affine or finite type, we used the \texttt{GAP} package \texttt{GBNP} to compute, for any length $\ell$, a basis of all components  of ${\rm nTL}(W)$ up to dimension $\ell$; equivalently, it gives us access to all FC elements up to Coxeter length $\ell$. 
  
  We used this data to verify our results for all classical types: we checked that the generating functions $W^{FC}(q)$ which we computed coincided with the Hilbert series given by the computer, for small values of the parameter $n$ and up to a large length $\ell$. This also gives confirmations that our descriptions of FC elements is correct for these types. For exceptional types, we used it to compute the generating functions themselves, as explained in Section~\ref{sub:affExcept}.

\subsection{Minimal periods}
\label{sub:min}

As we already pointed out, the periods obtained for classical affine types are not always the {\em minimal periods}. For instance, for  type $\aff{A}_{n-1}$, it was shown by Hanusa and Jones in~\cite{HanJon} that, when $n$ is prime, the corresponding growth sequence $(a^n_\ell)_{\ell\geq 0}$ is eventually constant. In the work~\cite{JouhetNadeau} by the second and third authors, the minimal periods for all cases is determined. In particular, in type $\aff{A}_{n-1}$, it is shown that the minimal period of $(a^n_\ell)_{\ell\geq 0}$ is $n$ unless  $n$ is a prime power $p^m$, in which case the minimal period is $p^{m-1}$. This generalizes the aforementioned result.

\subsection{Involutions and minuscule elements}

An element $w$ in a Coxeter group is an {\em involution} if and only if its set of reduced decompositions $\mathcal{R}(w)$ is {palindromic}, meaning that it is stable by taking the mirror images of its elements. For FC elements, this is equivalent to the fact that  $\H(w)$ is equal to its dual, which itself means that every word in $H_{\{s,t\}}$ is palindromic. For classical affine types, our characterizations of FC elements easily then specializes to involutions. The observation here is that in the path encoding of alternating heaps, palindromic words correspond to steps which are either up or down. From this it is possible to show that ultimate periodicity still holds in the enumerating sequences, and obtain all our results in the case of involutions for affine types, thereby extending the work of Stembridge from~\cite[Section 4]{St3}.

Another intersesting subset of FC elements is the set of {\em minuscule elements}, which are linked to representation theory. They were also studied by Stembridge in~\cite{Stem_Minuscule} who characterized their heaps by local conditions extending Proposition~\ref{prop:heaps_fullycom}. By using our description of FC heaps in the affine types, one can recognize among them which ones correspond to minuscule heaps and then study their enumerative properties. 

Both subsets will be the subject of the forthcoming article~\cite{BJN_families}. Let us add that a third one was introduced recently, namely the set of {\em cyclically fully commutative elements} (see~\cite{CyclicFC}). It is possible to describe explicitly the corresponding elements (or heaps) in finite and affine types and enumerate them; this is done in the work~\cite{Petreolle} of P\'etr\'eolle.

\subsection{Diagram representations of Temperley--Lieb algebras}
\label{sub:diagramTL}

We recall from the introduction that for a given $W$ of type $\Gamma$, the FC elements index naturally a basis of the (generalized) Temperley--Lieb algebra $\TL(W)$. On the other hand the usual Temperley--Lieb algebra of type $A$ is known to have a faithful representation as a {\em diagram algebra}. Such representations have since been extended to other types: $B$ and $D$ in~\cite{Green_TLBD}, $H$ in~\cite{Green_H}, $E$ in~\cite{Green_TLE}, which are finite dimensional algebras;  $\aff{A}$ in~\cite{FanGreen_Affine}, $\aff{C}$ in~\cite{ErnstDiagramI,ErnstDiagramII}, which are infinite dimensional algebras.

The procedure to obtain such a faithful representation is more or less always the same in the previously cited works: (1) Define a set $\mathcal{D}$ of (decorated) diagrams and a way to multiply them by some concatenation procedure; (2) determine a subset in $\mathcal{D}$ of elementary diagrams indexed by $S$, which satisfy the relations of $\TL(W)$; (3) Determine explicitly the subspace generated by the elementary diagrams, say $\mathcal{D}'$; (4) Prove that the surjective morphism $\TL(W)\to\mathcal{D}'$ thus obtained is injective.

It is these steps (3) and (4) that can be greatly simplified thanks to our global approach to fully commutative elements, as will be seen in the forthcoming work~\cite{BJN_FCdiagrams}, where we plan to extend such diagram algebras to the remaining classical affine types $\aff{B}$ and $\aff{D}$.

\subsection{General Coxeter groups}
\label{sub:gencox}
 Though we focused solely on affine and finite Coxeter groups in this work, FC elements are defined for any Coxeter group in Definition~\ref{defi:FC}. It is natural to ask how to extend our results to more general groups. This is done by the third author in~\cite{Nadeau} for two such extensions.
 
  First, it settles the problem which Corollary B in the introduction raises naturally: can one classify Coxeter groups $W$ having an ultimately periodic growth sequence? It turns out that there are only two such groups which are not affine. Second, it is shown in~\cite{Nadeau} that, for any $W$, the FC growth sequence satisfies a linear recurrence relation with constant coefficients. Equivalently, the generating function $W^{FC}(q)$ is always a rational function. Both of these results have direct consequences regarding the growth of generalized Temperley--Lieb algebras (see Section~\ref{sub:TL}).

\subsection{Further questions}
\label{sub:further}

 It would be interesting to explore other statistics on the sets $W^{FC}$ which can be studied naturally on heaps. An example would be the sets of left and right descents, which are defined for any Coxeter group: for a FC element $w$, these descents correspond to the minimal and maximal elements of  $\H(w)$.
\medskip

 Affine Coxeter groups get their name from the geometric representation of Coxeter groups; we refer to~\cite{Bourbaki} or~\cite{Humphreys} for details. In brief, elements of $W$ correspond bijectively to the regions (called \emph{alcoves}) of a certain regular tiling of $\mathbb{R}^n$. If $C_0$ is the alcove of the identity of $W$ and is fixed, then the length of $w$ corresponds to the distance from $C_0$ to $C$ (here the distance is measured in the minimum number of pairs of adjacent alcoves that one must encounter between $C_0$ and $C$, where two alcoves are adjacent if they are separated by a single hyperplane). 
 
 The alcoves for the affine group $\aff{G}_2$ are depicted in Figure~\ref{fig:g2}, the colored ones corresponding to FC elements. It is easy to give a geometric criterion for the location of alcoves for FC elements. It should be possible to use these geometric representations to obtain alternative proofs of our results. For instance, understanding the periodicity of the growth sequence from this point of view would be very interesting, especially if this can be done in a uniform manner.

\begin{figure}[!ht]
\begin{center}
\includegraphics{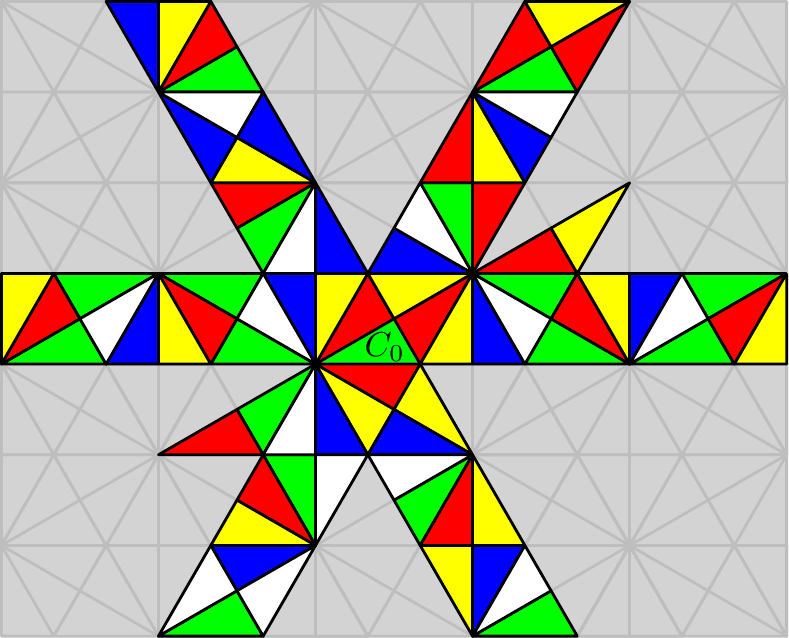}
\caption{\label{fig:g2} Fully commutative alcoves in type $\aff{G}_2$.}
\end{center}
\end{figure}

\end{document}